%% file: FFHKComp_arxiv.tex
\documentclass[11pt,utf8,notitlepage]{article}

\usepackage{amsmath}
\usepackage{mathtools}
\usepackage{amsthm} 
\usepackage{amssymb}
\usepackage{mathrsfs}
\usepackage{stmaryrd} 
\usepackage{commath}
\usepackage{bbm} 
\usepackage[perpage]{footmisc} 
\usepackage{tikz-cd}
\usepackage{enumitem}

\DeclareTextFontCommand{\myemph}{\bfseries\em}

\input{Notations-min.tex}
\newtheorem{thm}{Theorem}[section]
\newtheorem{pro}[thm]{Proposition}
\newtheorem{cor}[thm]{Corollary}
\newtheorem{lem}[thm]{Lemma}

\theoremstyle{definition}
\newtheorem{df}[thm]{Definition}
\newtheorem{eg}[thm]{Example}

\newtheorem{rmk}[thm]{Remark}

\makeatletter
\newcommand{\proofpart}[2]{%
  \par
  \addvspace{\medskipamount}%
  \noindent\underline{\bf Step #1: #2}\par\nobreak
  \addvspace{\smallskipamount}%
  \@afterheading
}
\makeatother

\numberwithin{equation}{section}

\usepackage[
backend=biber,
style=alphabetic,
sorting=nyt,
hyperref,
date=year,
backref,
backrefstyle=two,
maxbibnames=99,
maxcitenames=99,
maxalphanames=99,
doi=true,
isbn=false,
url=false,
useprefix=true 
]{biblatex}
\DeclareBibliographyAlias{misc}{online}

\DeclareFieldFormat{title}{\mkbibemph{#1}} 

 \renewbibmacro{in:}{} 

 \renewbibmacro*{journal+issuetitle}{%
   \usebibmacro{journal}%
  \setunit*{\addspace}%
  {\printtext{Vol. {\textbf{\printfield{volume}}}}}
  \setunit{\addspace}%
  \usebibmacro{issue+date}%
  \setunit{\addcomma\space}%
    \iffieldundef{number}
    {}
    {\printtext{No. \printfield{number}}}%
    \setunit{\addcomma\space}%
  \newunit}

\AtEveryBibitem{
  \ifentrytype{online}{
    \clearfield{doi}
    \clearfield{url}
  }{}
}
\usepackage{xcolor}
\definecolor{terblue}{HTML}{5371C6}
\definecolor{haoblue}{HTML}{83CCEB}
\definecolor{terpink}{HTML}{FF9497}

\usepackage[colorlinks=true,linkcolor=terblue,citecolor=haoblue,urlcolor=terpink]{hyperref}

\usepackage{geometry}
\usepackage{indentfirst} 
\usepackage{afterpage}

\usepackage{titlesec}
\usepackage[nottoc]{tocbibind} 
\usepackage[titles]{tocloft} 
\usepackage[T1]{fontenc}
\usepackage[utf8]{inputenc}
\usepackage{tgschola}
\usepackage[mathcal]{euscript}

\usepackage{fancyhdr}
\titleformat{\section}[block]
{\normalfont\Large\bfseries\filcenter}
{\thesection}
{8pt}{}

\titleformat{\subsubsection}[block]
{\bfseries\large\normalsize}
{}
{0pt}{}

\title{\LARGE \textbf{Comparison of Hyodo-Kato and de Rham Fargues-Fontaine Cohomology Theories}} 

\author{Kaixing Cao}
\date{\today}
\newtheorem*{conj}{Conjecture}
\begin{document}

\maketitle
\begin{abstract}
  We prove that, for adic \'{e}tale motives over $\mathbb{C}_p$, the vector bundles on the Fargues-Fontaine curve arising from their Hyodo-Kato cohomology coincide with their de Rham-Fargues-Fontaine cohomologies, where the latter provides an overconvergent refinement of crystalline vector bundles, albeit constructed on the generic fiber. This equivalence is established in the setting of symmetric monoidal $\infty$-categories and respects the full motivic structure. Furthermore, we enrich both realizations with Galois actions, yielding $G_{\breve{\Q}_{p}}$-equivariant solid quasi-coherent sheaves on the Fargues-Fontaine curve; in this equivariant context, the comparison isomorphism becomes canonical. As an application, we show that the de Rham-Fargues-Fontaine cohomology of any smooth quasi-compact rigid analytic variety over $\mathbb{C}_p$ admits a finite slope-increasing filtration.
\end{abstract}

\tableofcontents\thispagestyle{empty}
\section{Introduction}

Let $K$ be a complete discrete valued field of characteristic zero with perfect residue field $k$ of characteristic $p>0$. We fix an algebraic closure $\bar{K}$ of $K$ and let $C=\hat{\bar{K}}$ denote its completion. Then the residue field of $C$, denoted by $\bar{k}$, is an algebraic closure of $k$.

\subsection{Background and Motivation}
Bhatt, Morrow and Scholze studied various $p$-adic cohomology theories in \cite{BMS18}. We begin by recalling the geometrization of certain results in \cite{BMS18} from the perspective of the \myemph{Fargues-Fontaine curve} $\mathbf{FF}:= \mathbf{FF}_{C^{\flat}, \Q_p}$, where $C^{\flat}$ is the tilt of $C$. Let $\mathfrak{X}$ be a proper smooth formal scheme over $\mathcal{O}_C$. We consider the following cohomology theories:
\begin{itemize}
\item the $p$-adic \'{e}tale cohomology $\rgama{\text{\'{e}t}}(\mathfrak{X}_C, \Z_p)$ defined on the rigid generic fiber $\mathfrak{X}_C$;
\item the crystalline cohomology $\rgama{\mathrm{crys}}(\mathfrak{X}_{\bar{k}}/W(\bar{k}))$ defined on the special fiber $\mathfrak{X}_{\bar{k}}$;
\item a new cohomology theory $\rgama{\mathrm{crys}}(\mathfrak{X}_C/B_{\mathrm{dR}}^{+})$ defined on the rigid generic fiber.
\end{itemize}
Some of the comparison results in \cite{BMS18} yield a modification $\mathcal{E}\ets \to \mathcal{E}_{\mathrm{crys}}$ of vector bundles on the Fargues-Fontaine curve, where
\[
\mathcal{E}\ets:= H^i\ets(\mathfrak{X}_C,\Z_p) \otimes_{\Z_p} \mathcal{O}_{\mathbf{FF}}, \quad \mathcal{E}_{\mathrm{crys}}:= \mathcal{E}(H^i_{\mathrm{crys}}(\mathfrak{X}_{\bar{k}}/W(\bar{k})[1/p]),\varphi)
\]
and the functor $\mathcal{E}$ is Fargues-Fontaine's functor from $\varphi$-modules over $\bar{k}$ to vector bundles on $\mathbf{FF}$. More precisely, the crystalline vector bundle $\mathcal{E}_{\mathrm{crys}}$ is obtained from the \'{e}tale vector bundle $\mathcal{E} \ets$ by modifying it at the distinguished point $x_C$ determined by Fontaine's map $\theta \colon A_{\mathrm{inf}} \to \mathcal{O}_C$. This modification at $x_{C}$ is prescribed by the crystalline $B_{\mathrm{dR}}^+$-cohomology $H^i_{\mathrm{crys}}(\mathfrak{X}_C/ B_{\mathrm{dR}}^{+})$. In other words, the crystalline cohomology can be reconstructed from two cohomology data on the rigid generic fiber $\mathfrak{X}_{C}$: the $p$-adic \'{e}tale cohomology and the crystalline $B_{\mathrm{dR}}^+$-cohomology.

Together with the work of Colmez and Nizioł \cite{CN17,Niz19}, these results motivate the following conjecture, which seeks a direct cohomological construction for $\mathcal{E}_{\mathrm{crys}}$ on the rigid generic fiber:

\begin{conj}[{\cite[Conjecture 1.13]{faricm}, \cite[Conjecture 6.4]{Schicm}}]
  Let $\rigmot(C)$ be the $\infty$-category of adic \'{e}tale motives over $C$. There exists a motivic realization functor defined on compact rigid motives
  \[
\rgama{\mathrm{FF}} \colon \rigmot(C)_{\omega} \to \mathbf{Perf}(\mathbf{FF})
\]
sending each smooth quasi-compact rigid analytic variety $X$ over $C$ to perfect complexes $\rgama{\mathrm{FF}}(X)$ on the Fargues-Fontaine; in particular, each cohomology $\mathcal{H}^i_{\mathbf{FF}}(X)$ is a vector bundle on the Fargues-Fontaine curve. It is expected to satisfy the following properties:
\begin{enumerate}[label=(\arabic*)]
\item\label{conj1} the pullback $x_C^{*} \rgama{\mathrm{\mathbf{FF}}}(X)$ to the distinguished point is isomorphic to the overconvergent de Rham cohomology $\rgama{\mathrm{dR}}^{\dagger}(X)$;
  
\item\label{conj2} the complete stalk of $\rgama{\mathrm{FF}}(X)$ at $x_{C}$ recovers the overconvergent $B_{\mathrm{dR}}^+$-cohomology;
  
\item\label{conj3} in the case of semistable reduction, it agrees with the vector bundle associated to log-crystalline cohomology on the special fiber.
\end{enumerate}
From now on, we refer to any such cohomology theory as the \myemph{Fargues-Fontaine cohomology}.
\end{conj}

\begin{rmk}
\begin{enumerate}
\item Such a cohomology theory can in fact be extended to the entire category of rigid motives via Clausen and Scholze’s condensed mathematics. This yields a motivic realization functor on the whole category of motives
\[
\rgama{\mathrm{FF}} \colon \rigmot(C) \to \qcoh(\mathbf{FF})
\]
valued in solid quasi-coherent sheaves on the Fargues-Fontaine curve, in the sense of \cite{Andreqcoh21}.

\item All cohomology theories appearing in the conjecture are overconvergent since the category $\rigmot(C)$ serves as a model for mixed motives that allows one to treat non-proper spaces.
\end{enumerate}
\end{rmk}


\subsection{Hyodo-Kato Cohomology on the Rigid Generic Fiber}
The part \ref{conj3} of the conjecture asks for a comparison with log-crystalline cohomology, which is defined on the special fiber. However, Colmez and Nizioł, adapting a construction of Beilinson for algebraic varieties in \cite{Bei13}, introduced a Hyodo-Kato cohomology theory for rigid analytic spaces in \cite{CN20}. When $X$ is a rigid analytic space with semistable model $\mathfrak{X}$ over $\mathcal{O}_C$, their local-global compatibility result shows that the Hyodo-Kato cohomology $\rgama{\mathrm{HK}}(X)$ is isomorphic, as a $(\varphi,N)$-module, to the log-crystalline cohomology $\rgama{\mathrm{crys}}(\mathfrak{X}^0_{\mathcal{O}_C/p}/ \mathcal{O}_{\breve{K}}^0)[1/p]$. This leads to a refinement of part \ref{conj3} of the conjecture:
\begin{enumerate}[start=3,label={(\arabic*')}]
\item\label{conj4} the Fargues-Fontaine cohomology agrees with the overconvergent Hyodo-Kato cohomology.
\end{enumerate}

Since the Fargues-Fontaine cohomology is expected to be motivic, it is natural to formulate a comparison with the motivic (overconvergent) Hyodo-Kato cohomology: fixing a pseudo-uniformizer $\varpi \in \mathcal{O}_C$, we have the \myemph{$\boldsymbol{\varpi}$-Hyodo-Kato realization functor}
\begin{equation}
  \label{eq:intro-HK-varpi}
\rgama{\mathrm{HK}}^{\varpi} \colon \rigmot(C) \to \dcat_{(\varphi,N)}(\breve{K}),
\end{equation}
where $\breve{K}=\frc W(\bar{k})$. The canonical choice is $\varpi=p$, in which case we simply refer to it as the \myemph{Hyodo-Kato realization}, denoted by $\rgama{\mathrm{HK}}$.

\begin{rmk}
  \label{rmk:intro-HK-pse}
  As observed in \cite[Remark 4.18]{BGV25} and \cite[Remark 3.5]{BKV25}, different choices of pseudo-uniformizers affect only the behaviors of the monodromy operator on the Hyodo-Kato cohomology.
\end{rmk}


\subsection{The De Rham-Fargues-Fontaine Cohomology}

The main goal of this paper is to show that the de Rham-Fargues-Fontaine cohomology
\[
\rgama{\mathrm{FF}} \colon \rigmot(C) \simeq \rigmot(C^{\flat}) \to \qcoh(\mathbf{FF})
\]
constructed by Le Bras and Vezzani in \cite{LBV23}, qualifies as a Fargues-Fontaine cohomology; that is, it satisfies all properties stated in the conjecture above. Here the first equivalence is the motivic tilting equivalence established in \cite{Vez19a}.

To this end, since Le Bras and Vezzani have already shown that this cohomology theory satisfies parts \ref{conj1} and \ref{conj2} of the conjecture, it remains to prove the comparison with the Hyodo–Kato realization \eqref{eq:intro-HK-varpi}, namely part \ref{conj4}:

\begin{thm}[Theorem \ref{thm:FFHK-comp}, Remark \ref{rmk:indpend-psu}]
  \label{thm:introcomp}
  We have a monoidal equivalence $\rgama{\mathrm{FF}} \simeq \mathcal{E}_N \circ \rgama{\mathrm{HK}}^{\varpi} \colon \rigmot(C) \to \qcoh(\mathbf{FF})$ in $\calg(\prstc)$. Here $\mathcal{E}_N \colon \dcat_{(\varphi,N)}(\breve{K}) \to \qcoh(\mathbf{FF})$ extends the Fargues-Fontaine's functor from $(\varphi,N)$-modules over $\breve{K}$ to vector bundles over the Fargues-Fontaine curve.
\end{thm}

The proof of Theorem \ref{thm:introcomp} involves several crucial aspects. First, one has the key identifications (\cite[Theorem 3.7.21, Theorem 3.3.3]{AGV22})
\[
\rigmot(C) \simeq \rigmot_{\mathrm{gr}}(C) \simeq \module{\chi {\mathbbm 1}} (\agmot(\bar{k}))
\]
where $\chi {\mathbbm 1} \simeq {\mathbbm 1} \oplus {\mathbbm 1}(-1)[-1]$ is the cohomological motive of $\mathbb{G}_m$, arising from the fixed pseudo-uniformizer $\varpi \in \mathcal{O}_C$. Under this identification, the realization functors become functors from $\module{\chi {\mathbbm 1}}(\agmot(\bar{k}))$ to $\qcoh(\mathbf{FF})$. However, the mapping space
\[
\map(\module{\chi {\mathbbm 1}}(\agmot(\bar{k})), \qcoh(\mathbf{FF}))
\]
is not connected. Fortunately, both realization functors can be compared to the rigid realization functor (Proposition \ref{pro:HK-comp-rig} and Proposition \ref{pro:drFF-rig}). More precisely, let
\[
\xi \colon \agmot(\bar{k}) \to \rigmot(C)
\]
be the \myemph{Monsky-Washinitzer functor} defined in \cite{MWrig} (see also \eqref{eq:MWfunctor}). We have the following comparisons with rigid cohomology
\[
\mathcal{E}_N \circ \rgama{\mathrm{HK}} \circ \xi \simeq \mathcal{E} \circ \rgama{\mathrm{rig}} \simeq \rgama{\mathrm{FF}} \circ \xi
\]
in $\calg(\prstc)$. Therefore, these realization functors not only lie in $\calg(\prstc)$, but also in the category $\calg(\prstc)_{\agmot(\bar{k})/-}$ through their comparisons with the rigid cohomology. The refined mapping space
\[
\map_{\calg(\prstc)_{\agmot(\bar{k})/-}} \left( \rigmot(C), \qcoh(\mathbf{FF}) \right),
\]
particularly its set of connected components, can then be analyzed in terms of extension groups of vector bundles over the Fargues-Fontaine curve; see Proposition \ref{pro:mapsp-functors}.

\subsubsection*{Galois Enrichment for Motives and the Uniqueness of Comparisons}

Although we have a monoidal equivalence $\rgama{\mathrm{FF}} \simeq \mathcal{E}_N \circ \rgama{\mathrm{HK}}$ in $\calg(\prstc)$ obtained via a computation of the relevant mapping space, the equivalence is not canonical. Indeed, the space of such monoidal equivalences is large; see Remark \Ref{rmk:monoidal-eq-noncan}. As a result, there is no canonical comparison between these two realization functors.

To obtain the canonical comparison, we must incorporate the Galois action. More precisely, let $G_{\breve{K}}$ be the absolute Galois group of $\breve{K}$. Then both realization functors above can be upgraded to take values in the $\infty$-category of $G_{\breve{K}}$-equivariant solid quasi-coherent sheaves on $\mathbf{FF}$. That Hyodo-Kato cohomology naturally yields $G_{\breve{K}}$-equivariant vector bundles on the Fargues-Fontaine is already known from \cite{thecurve}.

We now explain how to equip the de Rham-Fargues-Fontaine cohomology with a $G_{\breve{K}}$-action. This crucial point is to define the $G_{\breve{K}}$-enrichment of adic \'{e}tale motives over $C$. Let $M$ be an adic \'{e}tale motive over $C$. Then it admits a model $\bar{M}$ over $\breve{K}$; that is, there exists an adic \'{e}tale motive $\bar{M}$ over $\breve{K}$ such that $M \simeq \bar{M}_{C}$. This fact follows from \cite[Theorem 3.7.21]{AGV22} and \cite[Proposition 3.23]{BKV25}. Thus, the identification $M \simeq \bar{M}_{C}$ equips $M$ with a natural $G_{\breve{K}}$-action, and hence its de Rham Fargues-Fontaine cohomology inherits a $G_{\breve{K}}$-action. This gives rise to the Galois-refined realization functor
\[
\rgama{\mathrm{FF}}^{\mathrm{ari}} \colon \rigmot(C) \to \qcoh(\mathbf{FF})^{\htcat G_{\breve{K}}}
\]
of the de Rham-Fargues-Fontaine realization. Here $\qcoh(\mathbf{FF})^{\htcat G_{\breve{K}}}$ is the $\infty$-category of $G_{\breve{K}}$-equivariant solid quasi-coherent sheaves on $\mathbf{FF}$; see \eqref{eq:G-on-qcFF}.

With these enhancements, we obtain a unique comparison equivalence:

\begin{thm}[Theorem \ref{thm:galFF-rig}]
There is a unique monoidal comparison equivalence $\rgama{\mathrm{FF}}^{\mathrm{ari}} \simeq \mathcal{E}_N^{\mathrm{ari}} \circ \rgama{\mathrm{HK}}$ in $\calg(\prstc)$, which is compatible with the monoidal comparison equivalences relating each of these two realization functors to the realization functor of the rigid cohomology on the special fiber.
\end{thm}

\subsubsection*{A New Filtration for Vector Bundles}

As a consequence of these comparison results, we deduce that the de Rham-Fargues-Fontaine realization factors through the weight complex functor (see \cite{Sos19,Aokwt,BGV25}) on the full subcategory $\rigmot(C)_{\omega}$ of compact motives. In particular, this yields a convergent spectral sequence computing the de Rham-Fargues-Fontaine cohomology of compact motives over $C$:

\begin{cor}[Corollary \ref{cor:fil-vbFF}, Corollary \ref{cor:fil-FFcoh-space}]
  Let $M$ be a compact motive (e.g, associated motives of smooth quasi-compact rigid analytic varieties) over $C$. Define
\[
\mathcal{H}_{\mathrm{FF}}^i(M):= H_{-i} \rgama{\mathrm{\mathrm{FF}}}(M^{\vee}).
\]
  Then there is a convergent (cohomological) spectral sequence starting from the first page and degenerates at the second page:
  \[
E_1^{pq} = \mathcal{H}^q_{\mathrm{FF}}(W^pM) \Rightarrow \mathcal{H}^{p+q}_{\mathrm{FF}}(M)
\]
where $W^{\bullet} M$ is the weight cochain complex of $M$. Moreover, the $i$-th graded piece of the induced finite (increasing) filtration on $\mathcal{H}^n_{\mathrm{FF}}(M)$ has a slope of $(i+n)/2$ and carries a natural $G_{\breve{K}}$-action. In particular, this is a new filtration on the vector bundle $\mathcal{H}_{\mathrm{FF}}^n(M)$ different from the Harder-Narasimhan filtration.
\end{cor}
\subsection{The Fargues-Fontaine Cohomology via the D\'{e}calage Functor}

Another version of Fargues-Fontaine cohomology, denoted by $\widetilde{\mathrm{R}\boldsymbol{\Gamma}}_{\mathrm{FF}}$, is defined via the d\'{e}calage functor; see \cite{LB18,Bos23}. Le Bras showed in \cite{LB18} that this cohomology is defined on effective adic \'{e}tale motives.

To compare this version with the de Rham-Fargues-Fontaine cohomology in a motivic way, we first prove that this Fargues-Fontaine cohomology $\widetilde{\mathrm{R}\boldsymbol{\Gamma}}_{\mathrm{FF}}$, defined by Le Bras and Bosco, is also motivic:

\begin{lem}[Proposition \ref{pro:BI-motivic}]
  \label{pro:intro-LBB-motivic}
  The functor sending each smooth dagger space $X$ over $C$ to a solid quasi-coherent sheaf $\widetilde{\mathrm{R}\Gamma}_{\mathrm{FF}}(X)$, defined \cite[Definition 6.16]{Bos23}, extends to a functor
\[
\widetilde{\mathrm{R}\boldsymbol{\Gamma}}_{\mathrm{FF}} \colon \rigmot(C)\simeq \rigmot^{\dagger}(C) \to \qcoh(\mathbf{FF})
\]
in $\calg(\prstc)$, where the first equivalence is given in \cite[Theorem 4.23]{MWrig}.
\end{lem}

This motivic extension enables a comparison between the two versions of Fargues–Fontaine cohomology via their respective relations to Hyodo–Kato cohomology:

\begin{pro}[Proposition \ref{pro:comp-FFcoh}]
There is a (non-canonical) monoidal equivalence $\rgama{\mathrm{FF}} \simeq \widetilde{\mathrm{R}\boldsymbol{\Gamma}}_{\mathrm{FF}} $ in $\calg(\prstc)$.
\end{pro}

{%
\paragraph{Notation and conventions.}
In this paper, we let $K$ be a complete discrete valued field over $\Q_{p}$ with a perfect residue field $k$. We fix an algebraic closure $\bar{K}$ of $K$ and let $C$ denote its completion, whose residue field is denoted by $\bar{k}$. As usual, we let $C^{\flat}$ denote the tilting of $C$, and fix a pseudo-uniformizer $p^{\flat}$ such that its first coordinate is $p$, i.e., $p^{\flat,\sharp}=p$. Let $K_{0}:=\frc W(k)$ and $\breve{K}:= \frc W(\bar{k})$.

For $\infty$-categories, we will follow notions in \cite{HTT,HA}. More precisely, we will let $\prl$ denote the $\infty$-category of presentable $\infty$-categories and left adjoint functors, and its subcategory of compactly generated $\infty$-categories with compact-preserving functors will be denoted by $\prlc$. There are two other variants: subcategories $\prst$ (resp. $\prstc$) of $\prl$ (resp. $\prlc$) of those also being stable. These categories underlie a symmetric monoidal structure, and their commutative algebra objects will be denoted by $\calg(\prl)$ (resp. $\calg(\prlc)$, $\calg(\prst)$, and $\calg(\prstc)$), whose objects are presentable (resp. compactly generated, stable presentable, stable compactly generated) symmetric monoidal $\infty$-categories for which the tensor products preserve small colimits in each variable. Finally, we will denote by $\mathrm{map}_{\mathcal{C}}(X,Y)$ or simply $\mathrm{map}(X,Y)$ the mapping spectrum from $X$ to $Y$ in an $\infty$-category $\mathcal{C}$.

All the motives are \'{e}tale motives with rational coefficients. Throughout the main text, we adopt the homological convention exclusively.
}%

{%
\paragraph{Acknowledgments.}%
The author would like to sincerely thank his advisor, Alberto Vezzani, for suggesting the topic of this paper and for many helpful conversations throughout the development of the work. He also wishes to thank Arthur-César Le Bras for valuable discussions during the author's visit, which helped clarify parts of the exposition concerning the Fargues-Fontaine curve.
}

\section{Adic \'{E}tale Motives and the Fargues-Fontaine Curve}
\label{sec:mot-FF}
In this section, we recall the theory of adic \'{e}tale motives and the Fargues-Fontaine curve from \cite{AGV22,LBV23}. In \S \ref{subsec:rigmot}, we review how adic \'{e}tale motives on the generic fiber can be related to motives on the special fiber, a perspective that plays a crucial role in the proof of the main comparison theorem. In \S \ref{subsec:FFmot}, we recall from \cite{LBV23} the spreading-out results for motives on the Fargues-Fontaine curve and compare the spreading behavior at $0$ and at $\infty$. Finally, we recollect the theory of solid quasi-coherent sheaves on the Fargues-Fontaine curve, with particular emphasis on their relationship with $\varphi$-modules over $\bar{k}$.
\subsection{Adic \'{E}tale Motives}
\label{subsec:rigmot}

Let $S$ be an admissible adic space (or more generally a rigid analytic space) over $\spa(\Z_p)$, as defined in \cite[Definition 2.1]{LBV23}. Following \cite{Vez19,LBV23}, the $\infty$-category of \'{e}tale sheaves with $\Q$-coefficients $\shv \ets(\mathbf{RigSm}/S, \dcat(\Q))$ gives rise to the $\infty$-category of \myemph{adic \'{e}tale motives} over $S$, denoted by $\rigmot(S):= \mathbf{RigSH}\ets(S, \Q)$ in the conventions of \cite{AGV22}. It is a compactly generated symmetric monoidal $\infty$-category admitting the formalism of six operators. We refer to \cite[\S 2]{LBV23} and \cite{AGV22} for further details.

Now assume $S$ is the rigid generic fiber of a formal scheme $\mathfrak{S}$. Taking special fibers and rigid generic fibers induces two functors between motives:
\[
(-)_{\sigma} \colon \fagmot(\mathfrak{S}) \to \agmot(\mathfrak{S}_{\sigma}), \quad (-)_{\eta} \colon \fagmot(\mathfrak{S}) \to \rigmot(S),
\]
where $\mathfrak{S}_{\sigma}$ denotes the special fiber of $\mathfrak{S}$, and $\fagmot(\mathfrak{S})$, $\agmot(\mathfrak{S}_{\sigma})$ are $\infty$-categories of \myemph{formal motives} over $\mathfrak{S}$, and \myemph{algebraic motives} over $\mathfrak{S}_{\sigma}$, respectively. These are constructed analogously to $\rigmot(S)$. In fact, the first functor above is an equivalence of $\infty$-categories by \cite[Theorem 3.1.10]{AGV22}. This yields a functor, called the \myemph{Monsky-Washnitzer functor},
\begin{equation}
  \label{eq:MWfunctor}
  \xi \colon \agmot(\mathfrak{S}_{\sigma}) \simeq \fagmot(\mathfrak{S}) \xrightarrow{(-)_{\eta}} \rigmot(S)
\end{equation}
which provides a powerful new link between rigid analytic and algebraic geometry, thanks to deep results from \cite{AGV22}:

\begin{pro}
  \label{pro:chi-unit-formula}
  Let $S$ and $\mathfrak{S}$ be as above.
  \begin{enumerate}
  \item The Monsky-Washnitzer functor \eqref{eq:MWfunctor} admits a right adjoint $\chi \colon \rigmot(S) \to \agmot(\mathfrak{S}_{\sigma})$ and induces a fully faithful functor
    \[
\tilde{\xi} \colon \module{\chi {\mathbbm 1}}(\agmot(\mathfrak{S}_{\sigma})) \to \rigmot(S),
\]
where ${\mathbbm 1}$ is the tensor unit of $\rigmot(S)$ and its image $\chi {\mathbbm 1}$ is a commutative algebra object in $\agmot(\mathfrak{S}_{\sigma})$.
    
  \item Suppose $S=\spa(K)$ and $\mathfrak{S}= \spf(\mathcal{O}_K)$. A choice of a pseudo-uniformizer yields an identification $\chi {\mathbbm 1} \simeq {\mathbbm 1}\bigoplus {\mathbbm 1}(-1)[-1]$, where ${\mathbbm 1}(-1)$ is the $(-1)$-Tate twist.
\end{enumerate}
\end{pro}
\begin{proof}
These are \cite[Theorem 3.3.3, Theorem 3.7.21, Corollary 3.8.32]{AGV22}.
\end{proof}

From now on, if $S$ is the spectrum of a complete non-archimedean field $K$, then we write $\rigmot(K)$ for $\rigmot(\spa(K))$. Similar notations apply to algebraic and formal motives when the base is affine.

\begin{df}
  \label{df:motgr}
Define $\rigmot_{\mathrm{gr}}(S)$ as the full subcategory of $\rigmot(S)$ spanned by the essential image of the fully faithful functor in Proposition~\ref{pro:chi-unit-formula}. We call it the $\infty$-category of \myemph{adic \'{e}tale motives with good models}.
\end{df}

\begin{rmk}
  \label{rmk:grmot}
\begin{enumerate}
\item The full subcategory $\rigmot_{\mathrm{gr}}(S)$ is also given by the full subcategory of $\rigmot(S)$ generated under small colimits by those motives of the form $M(\mathfrak{X}_{\eta})$, where $\mathfrak{X}$ runs through smooth $\mathfrak{S}$-formal schemes. However, it includes not only smooth $S$-admissible adic spaces with good reduction but also those with semi-stable reduction; see \cite[Proposition 3.29]{BKV25}.

\item In light of Proposition \ref{pro:chi-unit-formula}, we have an equivalence $\module{\chi {\mathbbm 1}}(\agmot(\mathfrak{S}_{\sigma})) \simeq \rigmot_{\mathrm{gr}}(S)$. In the special case where $S$ is associated with a complete algebraically closed non-archimedean field $K$ with residue field $k$, the inclusion of this subcategory category into $\rigmot(K)$ is itself an equivalence of $\infty$-categories; see \cite[Theorem 3.7.21]{AGV22}. Consequently, we obtain canonical equivalences
\[
\module{\chi {\mathbbm 1}}(\agmot(k)) \simeq \rigmot_{\mathrm{gr}}(K) \simeq  \rigmot(K).
\]
These identifications play a central role in establishing the forthcoming comparison results.
\end{enumerate}  
\end{rmk}

\subsection{Rigid Motives on the Fargues-Fontaine Curve}
\label{subsec:FFmot}
We now recall the spreading-out result for motives on the Fargues-Fontaine curve from \cite[\S 5]{LBV23}, which will yield a ``pullback'' functor $\rigmot(C^{\flat}) \to \rigmot(\mathrm{FF})$. Here $\mathbf{FF}:= \mathbf{FF}_{C^{\flat}, \Q_p}$ denotes the absolute adic Fargues-Fontaine curve over $\Q_p$. More precisely, it is the quotient of the Fargues-Fontaine's space $\mathcal{Y}_{(0,\infty)}:= \spa(A_{\mathrm{inf}}, A_{\mathrm{inf}})\setminus V(p[p^{\flat}])$ by the Frobenius action:
\[
\mathbf{FF}= \mathcal{Y}_{(0,\infty)}/\varphi^{\Z}.
\]

We next recall some relevant notations. There are two distinguished points $x_{0}$ and $x_{\infty}$ on $\mathbf{FF}$ corresponding to $C^{\flat}$ and $\breve{K}$, respectively. Let
\begin{align*}
  \mathcal{Y}_0&:= \mathcal{Y}_{[0,\infty)}=\spa(A_{\mathrm{inf}},A_{\mathrm{inf}})\setminus V([p^{\flat}])\\
  \mathcal{Y}_{\infty}&:=\mathcal{Y}_{(0,\infty]}= \spa(A_{\mathrm{inf}},A_{\mathrm{inf}})\setminus V(p),
\end{align*}
so that for $\star \in \{0,\infty\}$, $\mathcal{Y}_{\star}$ is an open neighborhood of $x_{\star}$. We adopt the following Frobenius conventions: for $x_0$, $\mathcal{Y}_{0}$ and $\mathcal{Y}_{(0,\infty)}$, we use the usual Frobenius endomorphism $\varphi_{0}= \varphi$; for $x_{\infty}$ and $\mathcal{Y}_{\infty}$, we set $\varphi_{\infty}= \varphi^{-1}$. Then, for $S \in \{x_0,x_{\infty}, \mathcal{Y}_0, \mathcal{Y}_{\infty}, \mathcal{Y}_{(0,\infty)}\}$, with the above Frobenius convention, we denote by $\rigmot(S)^{\varphi}$ (resp. $\rigmot(S)^{\varphi_{\omega}}$) the equalizer of $\mathrm{Id}$ and $\varphi^{*}$ in $\calg(\prl)$ (resp. $\calg(\prlc)$).

Using the semi-separatedness property of motives together with the Frobenius trick on the Fargues-Fontaine curve, one then obtains the following spreading-out results from either $x_0$ or $x_{\infty}$:

\begin{pro}
  \label{pro:sp-out-curve}
  Let $\star \in \{0,\infty\}$ and $r \in \Q_{>0}$.
  \begin{enumerate}
  \item The pullbacks along the closed embeddings $x_{\star} \hookrightarrow \mathcal{Y}_{\star}$ induce equivalences in $\calg(\prlc)$:
  \[
\rigmot(\mathcal{Y}_{\star})^{\varphi_{\omega}} \xrightarrow{\simeq} \rigmot(x_{\star})^{\varphi_{\omega}}.
\]

\item The pullback defines an equivalence in $\calg(\prlc)$
  \[
\rigmot(\mathbf{FF}) \simeq \rigmot(\mathcal{Y}_{(0,\infty)})^{\varphi} \simeq \rigmot(\mathcal{Y}_{(0,\infty)})^{\varphi_{\omega}}.
\]
\end{enumerate}
\end{pro}
\begin{proof}
The proof of \cite[Proposition 5.3]{LBV23} works in both cases.
\end{proof}

To define the spreading-out functor, we require a Frobenius enrichment functor. Since $\spa(C^{\flat})$ is of characteristic $p>0$, there is a functor
\[
\rigmot(C^{\flat}) \to \rigmot(C^{\flat})^{\varphi_{\omega}}
\]
defined by the relative Frobenius (see \cite[Corollary 2.26]{LBV23}, \cite[Corollary 2.9.11]{AGV22}). Now using Proposition \ref{pro:sp-out-curve} and the pullback $\rigmot(\mathcal{Y}_{[0,\infty)}) \to \rigmot(\mathcal{Y}_{(0,\infty)})$, we get a functor
\begin{equation}
  \label{eq:spreading-0}
    \mathcal{D}_0 \colon \rigmot(C^{\flat}) \to \rigmot(\mathbf{FF})
\end{equation}
defined by
\[
\begin{tikzcd}
\rigmot(C^{\flat}) \arrow[r, "\mathcal{D}_0"] \arrow[d]                & \rigmot(\mathbf{FF}) \arrow[d, "\simeq"]                        \\
\rigmot(C^{\flat})^{\varphi_{\omega}} \arrow[d, "\simeq"']             & {\rigmot(\mathcal{Y}_{(0,\infty)})^{\varphi}}                   \\
{\rigmot(\mathcal{Y}_{[0,\infty)})^{\varphi_{\omega}}} \arrow[r, hook] & {\rigmot(\mathcal{Y}_{[0,\infty)})^{\varphi}} \arrow[u, "j^*"']
\end{tikzcd}
\]
in $\calg(\prlc)$. We will refer to it as the \myemph{spreading-out functor} from $0$.

On the other hand, we can also spread out from $x_{\infty}$. However, a crucial ingredient in defining the spreading-out functor from $x_0$ is the Frobenius enrichment, which arises naturally in characteristic $p>0$. In characteristic $0$, we instead employ an alternative method to obtain a Frobenius equivariant enrichment of motives, intended specifically for comparison purposes.

In fact, for algebraic motives, we also have the Frobenius enrichment:
\begin{equation}
  \label{eq:rel-frob}
 \varphi_{\mathrm{rel}} \colon \agmot(\bar{k}) \to \agmot(\bar{k})^{\varphi_{\omega}}
\end{equation}
Here, as before, $\agmot(\bar{k})^{\varphi_{\omega}}$ denotes the equalizer of $\mathrm{Id}$ and the Frobenius pullback, and this enrichment is defined via the relative Frobenius (\cite[\href{https://stacks.math.columbia.edu/tag/0CC6}{Section 0CC6}]{stacks-project}), see also \cite[Remark 4.28]{BGV25}, \cite[Proposition 6.3.16]{CD16}, and \cite[Theorem 2.9.7]{AGV22}. This construction remains valid when replacing $\bar{k}$ by any scheme of characteristic $p>0$.

Therefore, composed by the Monsky-Washinitzer functor (\ref{eq:MWfunctor}), we obtain a functor
\[
\agmot(\bar{k}) \xrightarrow{\varphi_{\mathrm{rel}}} \agmot(\bar{k})^{\varphi_{\omega}} \xrightarrow{\xi_{\breve{K}}} \rigmot(\breve{K})^{\varphi_{\omega}}.
\]
Using Proposition \ref{pro:sp-out-curve} and imitating the construction of $\mathcal{D}_{0}$, we get a functor
\[
\bar{\mathcal{D}}_{\infty} \colon \agmot(\bar{k}) \to \rigmot(\breve{K})^{\varphi_{\omega}} \to \rigmot(\mathbf{FF})
\]
in $\calg(\prlc)$. These allow us to compare spreading out from $0$ and $\infty$:

\begin{pro}
  \label{pro:comp-spread-0-inf}
  We have a monoidal equivalence $ \mathcal{D}_0 \circ \xi_{C^{\flat}} \simeq \bar{\mathcal{D}}_{\infty}$ in $\calg(\prlc)$.
\end{pro}
\begin{proof}
  We use the equivalence $\agmot(\bar{k}) \simeq \fagmot(W(\mathcal{O}_C^{\flat}))$ from \cite[Theorem 3.1.10]{AGV22}, where $W(\mathcal{O}_C^{\flat})$ is equipped with the $(p,[p^{\flat}])$-adic topology, here $p^{\flat} \in \mathcal{O}_C^{\flat}$ is a pseudo-uniformizer with $p^{\flat,\sharp}=p$. Then we have a commutative diagram (up to homotopy)
\[
\begin{tikzcd}
\agmot(\bar{k}) \arrow[r, "\varphi_{\mathrm{rel}}"]             & \agmot(\bar{k})^{\varphi_{\omega}} \arrow[r]                                       & \rigmot(\breve{K})^{\varphi_{\omega}} \arrow[r]                                       & \rigmot(\mathbf{FF}) \\
\fagmot(W(\mathcal{O}_C^{\flat})) \arrow[u, "\simeq"] \arrow[r] & \fagmot(W(\mathcal{O}_C^{\flat}))^{\varphi_{\omega}} \arrow[u, "\simeq"] \arrow[r] & {\rigmot(\mathcal{Y}_{[0,\infty]})^{\varphi_{\omega}}} \arrow[u] \arrow[ru] &                     
\end{tikzcd}
\]
where the top row from left to right is the functor $\bar{\mathcal{D}}_{\infty}$. Therefore, it suffices to show the bottom functor
\[
\fagmot(W(\mathcal{O}_C^{\flat})) \to \rigmot(\mathcal{Y}_{(0,\infty]})^{\varphi_{\omega}} \to \rigmot(\mathbf{FF})
\]
is equivalent to $\mathcal{D}_0 \circ \xi_{C^{\flat}}$. This is \cite[Proposition 5.11]{LBV23}.
\end{proof}

\subsection{\texorpdfstring{The Derived $\boldsymbol{\infty}$-Category of $\boldsymbol{\varphi}$-Modules}{The Derived Infinite-Category of phi-Modules}}
We recall the derived $\infty$-category of $\varphi$-modules and will relate them to vector bundles on the Fargues-Fontaine curve in the next subsection.

Let $l$ be a perfect field and $L_0=\frc W(l)$. Recall that a \myemph{$\boldsymbol{\varphi}$-module} over $l$ is a finite-dimensional $L_0$-vector space $D$ with a $\sigma$-semi-linear bijection $\varphi_D \colon D \to D$, where $\sigma \colon L_0 \to L_0$ is the Frobenius induced by the Frobenius of $l$. We will denote the bounded derived $\infty$-category of $\varphi$-modules over $l$ by $\bdcat_{\varphi}(L_0)$ and its Ind-completion by $\dcat_{\varphi}(L_0) \in \calg(\prstc)$.

\begin{eg}[Simple $\varphi$-Modules]
  \label{eg:isocry-twist}
  For every $n \in \Z$, we define the \myemph{Tate twist} $L_0(n)$ by the one-dimensional $L_{0}$-vector space $L_{0}$ with the twisted Frobenius\footnote{Our twist convention is dual to the standard one.} $\varphi_{L_0(n)}=p^{-n} \sigma$.

  More generally, let $\lambda \in \Q$ with the unique form $\lambda=m/n$ with $m\in \Z$, $n \in \Z_{>0}$ and $\mathrm{gcd}(m,n)=1$. We can define an $\varphi$-module $L_0(\lambda)$ over $l$ as follows: the underlying vector space is $L_{0}^n$ and the bijection $\varphi_{\lambda} \colon L_0(\lambda) \to L_{0}(\lambda)$ is given by
  \[
\varphi_{\lambda}(e_1,\dots,e_n) = (e_1,\dots,e_n) A_{\lambda}, \quad \text{ with } A_{\lambda} =       \begin{pmatrix}
        0&I_{n-1}\\
        p^{m}& 0
      \end{pmatrix}
    \]
    where $(e_1,\dots,e_{n})$ is the standard basis of $L_0^{n}$. It is not hard to see the characteristic polynomial of $\varphi_{\lambda}$ is $f_{\lambda}(T)=(-1)^n (T^n-p^m)$. We will refer to the rational number $\lambda$ as the \myemph{slope} of $L_0(\lambda)$.
  \end{eg}

  \begin{df}
  \label{df:phi-mod-weight}
\begin{enumerate}
\item If $l$ is a finite field with cardinality $q=p^{f}$, we say a $\varphi$-module $(D,\varphi_{D})$ over $L_0$ is \myemph{pure of weight $i$} if each generalized eigenvalue $\lambda$ of the $D$ is a $q$-Weil number of weight $i$, i.e., it is an algebraic number such that $\abs{\sigma(\lambda)}=q^{i/2}$ for every embedding $\sigma\colon \bar{\Q} \hookrightarrow \mathbb{C}$.
  
\item If $l= \bar{\mathbb{F}}_p$, we say a $\varphi$-module $(D, \varphi_D)$ over $L_0$ is \myemph{pure of weight $i$} if it has a model of pure weight $i$ over $\mathbb{F}_q$ for some $q=p^{f}$.
\end{enumerate}
\end{df}

We next study the relationship between slopes and weights of $\varphi$-modules (see Proposition \ref{pro:pure-isoc-half-slope}). They will be used to construct the new filtration on the de Rham-Fargues-Fontaine cohomologies; see Corollary \ref{cor:fil-vbFF}. Readers primarily interested in the comparison result may skip the remainder of this subsection.

\subsubsection{Weights and Slopes of $\boldsymbol{\varphi}$-Modules}
We fix an algebraic closure $\bar{l}$ of $l$. Let $\breve{L}$ (resp. $L_0$) be the fraction field of the ring of Witt vectors over $\bar{l}$ (resp. over $l$).

\begin{lem}
  \label{lem:ss-wt-slope}
  Assume $l= \mathbb{F}_q$ with $q=p^{f}$. Let $(D, \varphi)$ be a pure $\varphi$-module
  of weight $i$ over
  $l$. If the $\varphi$-module $\hat{D}= D \otimes_{L_0} \breve{L}$ is semistable of slope
  $\lambda$, i.e., direct sum of $\breve{L}(\lambda)$, then $\lambda=i/2$.
\end{lem}
\begin{proof}
Let $\hat{\varphi}$ be the Frobenius automorphism of $\hat{D}$,
i.e., $\hat{\varphi}= \varphi \otimes \sigma$. By the assumption, we
can find $\alpha \in D$ and $\mu \in L_0$ such that $\varphi^f\alpha= \mu \alpha$ and $\abs{\mu}=q^{i/2}$ for every embedding $\overline{ \Q }_p \hookrightarrow
\mathbb{C}$. Choose a basis $(e_1,\dots,e_d)$ of $D$ over $L_0$ such
that the matrix of $\hat{\varphi}$ under the induced basis
$(\hat{e}_1,\dots, \hat{e}_{d})$ is
\[
\operatorname{diag}(A_{\lambda}, A_{\lambda},\cdots, A_{\lambda})
\]
where $A_{\lambda}$ is the one in Example \ref{eg:isocry-twist}. We write $  \alpha=x_1e_1+\cdots +x_de_d$ with $x_i \in L_0$. Note that we have
\begin{align*}
\hat{\varphi}^f(\alpha \otimes 1)&= (\hat{e}_1,\dots, \hat{e}_d)
\operatorname{diag}(A_{\lambda}^f,\dots,A_{\lambda}^f)
\begin{pmatrix}
  \sigma^f(x_1)\\
  \vdots\\
  \sigma^f(x_d)\\
  \end{pmatrix}\\
  &=(\hat{e}_1,\dots, \hat{e}_d)
\operatorname{diag}(A_{\lambda}^f,\dots,A_{\lambda}^f)
\begin{pmatrix}
  x_{1}\\
  \vdots\\
  x_{d}
\end{pmatrix}
\end{align*}
since $x_i \in L_0$, hence $\sigma^f(x_i)=x_i$. On the other hand, we
have
\[
\hat{\varphi}^f(\alpha \otimes 1)= \mu (\alpha \otimes 1)=
(\hat{e}_1,\dots, \hat{e}_d) \mu
\begin{pmatrix}
  x_1\\
  \vdots\\
  x_d
\end{pmatrix}.
\]
Therefore, $\mu$ is also the eigenvalue of $A_{\lambda}$, and then we
have
\[
q^{\frac{i}{2}}= \abs{\mu}= p^{f\lambda}=q^{\lambda},
\]
hence $\lambda=i/2$ as desired.
\end{proof}

\begin{pro}
  \label{pro:pure-isoc-half-slope}
Assume $l$ is algebraically closed. If $(D, \varphi_{D})$ is a pure
$\varphi$-module over $l$ of weight $i$, then it is semistable of slope $i/2$.
\end{pro}
\begin{proof}
By the assumption, we can find a model $(V, \varphi_{V})$ of $(D,
\varphi_{D})$ over $\mathbb{F}_q$ for some $q=p^f$. Assume
\[
V \cong \bigoplus_{j=1}^r V(\alpha_j)
\]
is the slope decomposition as in \cite[Lemma 8.1.11]{BC09}, i.e., the base change of the decomposition into $L_0= \frc W(l)$ is the Dieudonné–Manin decomposition of $(D, \varphi_{D})$. Since $V$ is pure of weight $i$, so is
$V(\alpha_j)$ for every $j$. By Lemma \ref{lem:ss-wt-slope}, we can
know that $\alpha_j=i/2$ for every $j$, hence $D$ is semistable of slope $i/2$.
\end{proof}


\subsection{Solid Quasi-Coherent Sheaves on the Fargues-Fontaine Curve}
Following \cite{Andreqcoh21}, there exists an $\infty$-category of \myemph{solid quasi-coherent sheaves} on the Fargues-Fontaine curve, denoted by $\qcoh(\mathbf{FF})$. This category is constructed via using Clausen-Scholze's condensed mathematics; see \cite{CS1,CS2}. For convenience, we fix an implicit strong limit cardinal $\kappa$ and work with $\kappa$-condensed objects. As a result, the descent properties established \cite{Andreqcoh21} hold in $\calg(\prstc)$. In particular, we have $\qcoh(\mathbf{FF}) \in \calg(\prstc)$.

In this paper, a central role will be played by the stable full subcategory of perfect complexes $\mathbf{Perf}(\mathbf{FF})$. These perfect complexes are indeed strict by \cite[Proposition 2.6]{AL25}: each perfect complex $\mathcal{E} \in \mathbf{Perf}(\mathbf{FF})$ is quasi-isomorphic to a bounded complex of vector bundles.

We now focus on the computations of vector bundles over $\mathbf{FF}$ and how they are related to $\varphi$-modules over $\bar{k}$. Recall that we can cover $\mathcal{Y}_{(0,\infty)}$ by open subspaces $\mathcal{Y}_{I}$, as defined in \cite[\S 12.2]{berkeley}, where $I$ runs through compact sub-intervals in $(0,\infty)$ with rational endpoints. The global section $\Gamma(\mathcal{Y}_I, \mathcal{O})$ is given by $B_I$, as defined in \cite[\S 1.6, especially Example 1.6.3]{thecurve}. This shows that we have an isomorphism
\[
  \Gamma(\mathcal{Y}_{(0,\infty)}, \mathcal{O})=\lim_{I \subseteq (0,\infty)} B_I \simeq B.
\]
Here $B$ is obtained by inverting $p$ in the completion of $A_{\mathrm{inf}}[1/p, 1/[p^{\flat}]]$ with respect to a family of Gauss norms (c.f. \cite[\S 1.6]{thecurve}). Clearly, there is a natural map $\breve{K} \to B$ induced by the inclusion $W(\bar{k}) \to \mathcal{O}_C$ (or equivalently, $\bar{k} \to \mathcal{O}_{C^{\flat}}$). This yields a morphism of adic spaces:
\begin{equation}
  \label{eq:map-Y-breveK}
  e \colon \mathcal{Y}_{(0,\infty)} \to \spa(\breve{K})
\end{equation}
which is also compatible with Frobenius morphisms.

The pullback along the morphism $e$ in \eqref{eq:map-Y-breveK} defines a functor
\[
\mathcal{E} \colon \bdcat_{\varphi}(\breve{K}) \simeq \mathbf{Perf}(\spa(\breve{K}, \mathcal{O}_{\breve{K}}))^{\varphi} \to \mathbf{Perf}(\mathcal{Y}_{(0,\infty)})^{\varphi} \simeq \mathbf{Perf}(\mathbf{FF}),
\]
where the superscripts $\varphi$ mean to take the equalizer of $\varphi^{*}$ and $\mathrm{Id}$ in $\catinf$, and the first equivalence is \cite[Proposition 2.50]{BGV25} and the last equivalence follows from the analytic descent of perfect complexes in \cite[Theorem 5.3]{Andreqcoh21}. Since the perfect complexes are compact objects in $\qcoh(\mathbf{FF})$, as shown in \cite[Proposition 5.37 and Corollary 5.51.1]{Andreqcoh21}, this functor can extend to unbounded derived $\infty$-categories by taking Ind-completion:
\begin{equation}
  \label{eq:E-functor}
  \mathcal{E} \colon \dcat_{\varphi}(\breve{K}) \to \qcoh(\mathbf{FF})
\end{equation}

\begin{rmk}
  \label{rmk:factorization-e}
We use the natural map $\breve{K} \to B$ to define a $\varphi$-equivariant morphism $e \colon \mathcal{Y}_{(0,\infty)} \to \spa(\breve{K})$ in \eqref{eq:map-Y-breveK}. A similar argument shows that $\Gamma(\mathcal{Y}_{(0,\infty]}, \mathcal{O})=B^+$ (we adopt the notation in \cite[\S 1.10]{thecurve}). This gives a morphism $\bar{e} \colon \mathcal{Y}_{(0,\infty]} \to \spa(\breve{K})$. In fact, the morphism $e$ can be factorized into the composition of $\bar{e}$ with the pullback $\iota \colon \mathcal{Y}_{(0,\infty)} \to \mathcal{Y}_{(0,\infty]}$. In fact, the natural closed immersion $c \colon \spa(\breve{K}) \to \mathcal{Y}_{(0,\infty]}$ is a section of $\bar{e}$.
\end{rmk}

\begin{rmk}[GAGA]
  \label{rmk:FF-GAGA}
  The original construction of the Fargues-Fontaine curve is algebraic (e.g. in \cite{thecurve}) and we have the GAGA theorem for it, i.e., their category of coherent sheaves are equivalent (\cite{KL1}, \cite{Far18}, \cite{FS24}) via analytification.
\end{rmk}

By our construction, this is exactly the analytic version (using GAGA below) of the algebraic $\mathcal{E}$-functor defined in \cite[\S 8.2.3]{thecurve}. According to our conventions on Tate twists of $\varphi$-modules (in Example \ref{eg:isocry-twist}), we define
\begin{equation}
  \label{eq:FF-twist}
  \mathcal{O}_{\mathbf{FF}}(1):= \mathcal{E}(\breve{K}(1)).,
\end{equation}
which agrees with the usual notation on the Fargues-Fontaine curve. Accordingly, as is standard, we obtain all the twists $\mathcal{O}_{\mathbf{FF}}(\lambda)$ for $\lambda \in \Q$.

\begin{pro}
  \label{pro:curve-twist}
  Let $\lambda, \mu \in \Q$. Then we have the following computations about mapping spectra:
\begin{align*}
\pi_0 \mathrm{map} (\mathcal{O}_{\mathbf{FF}}(\lambda), \mathcal{O}_{\mathbf{FF}}(\mu))&=
{\begin{cases}
0  & \mu < \lambda \\
B^{\varphi^h=p^{d}}   & \mu-\lambda=d/h, d,h \in \Z_{\ge 0}, \gcd(d,h)=1
\end{cases}}\\
  \pi_{-1}  \mathrm{map}(\mathcal{O}_{\mathbf{FF}}(\lambda), \mathcal{O}_{\mathbf{FF}}(\mu))&=0 ,\quad                 \lambda\le \mu                                    
\end{align*}
\end{pro}
\begin{proof}
This follows from computations in the algebraic case (\cite[Proposition 5.6.23, Proposition 8.2.3]{thecurve}) and the GAGA theorem (cf. \cite[Proposition II.2.7]{FS24}).
\end{proof}

\section{Motivic Realization Functors and Comparisons}
The main goal of this section is to establish a comparison between the de Rham-Fargues-Fontaine cohomology and the Hyodo-Kato cohomology. We begin by recalling their construction from \cite{LBV23, BGV25}, and then compare them using properties of motives introduced in \S \ref{subsec:rigmot}. In the final part, we also compare another Fargues-Fontaine cohomology studied in \cite{LB18, Bos23} with the de Rham-Fargues-Fontaine cohomology in the motivic context.
\subsection{The Rigid Cohomology and the Hyodo-Kato Cohomology}
\label{subsec:HKcoh}
As in \cite{MWrig}, there is a direct definition for the rigid cohomology via the overconvergent de Rham cohomology:

Let $l$ be an arbitrary perfect field of characteristic $p>0$. Applying the overconvergent de Rham cohomology to the Monsky-Washnitzer functor \eqref{eq:MWfunctor}, we obtain a realization functor :
\[
\rgama{\mathrm{rig}} \colon \agmot(l) \xrightarrow{\xi} \rigmot(L_{0}) \xrightarrow{\rgama{\mathrm{dR},K_0}^{\dagger}} \dcat(L_{0}),
\]
where $L_0= \frc  W(l)$ and $\rgama{\mathrm{dR},K_0}^{\dagger}$ is the covariant overconvergent de Rham realization studied in \cite{MWrig, LBV23}. To be precise, the restriction of $\rgama{\mathrm{dR},K_0}^{\dagger}$ to compact motives is given by the overconvergent de Rham realization of the duals of these compact motives\footnote{They are dualizable by \cite{Rio05}.}, and then one can extend to the non-compact part by taking Ind-completion.

Furthermore, we can enhance this realization by taking value of $\varphi$-modules: as $\rgama{\mathrm{rig}}$ is compatible with the Frobenius pullbacks, we have
\[
\rgama{\mathrm{rig}}^{\varphi} \colon \agmot(l)^{\varphi_{\omega}} \to \mathcal{D}_{\varphi}(L_{0}).
\]
Composing with the relative Frobenius (\ref{eq:rel-frob}), we get an algebraic realization taking value of $\varphi$-modules, denoted by 
\begin{equation}
  \label{eq:rigcoh}
  \rgama{\mathrm{rig}} \colon \agmot(l) \to \dcat_{\varphi}(L_{0}),
\end{equation}
referred as the \myemph{rigid realization}.

We next use the rigid realization \eqref{eq:rigcoh} to give a quick definition of the motivic Hyodo-Kato realization defined in \cite{BGV25}. For this, we put $l= \bar{k}$. In the case, the rigid realization is
\[
\rgama{\mathrm{rig}} \colon \agmot(\bar{k}) \to \dcat_{\varphi}(\breve{K}).
\]
Note that the image of ${\mathbbm 1}(-1)$ under this realization is $\breve{K}(-1)$ (\cite[Remark 4.41, Lemma 4.30]{BGV25}). Then adding the monodromy operators gives the \myemph{(geometric) Hyodo-Kato realization}
\begin{equation}
  \label{eq:HK-coh}
  \rgama{\mathrm{HK},C} \colon \rigmot(C) \to \dcat_{(\varphi,N)}(\breve{K}).
\end{equation}
More precisely, if we choose the canonical pseudo-uniformizer $p \in \mathcal{O}_{C}$ and use the identifications in Remark \ref{rmk:grmot} ($2$), then there is an equivalence
\[
\rigmot(C) \simeq \agmot_N(\bar{k})
\]
of $\infty$-categories by \cite[Corollary 4.14]{BGV25}; here $\agmot_N(\bar{k})$ is obtained by adding monodromy operators in $\agmot(\bar{k})$. For further details, we refer to \cite{BGV25}, especially \cite[\S 2 and Corollary 4.14]{BGV25}.

\begin{rmk}
  \label{rmk:two-HK-comp}
  The Hyodo-Kato realization (\ref{eq:HK-coh}) agrees with the one in \cite[Definition 4.42]{BGV25}. This follows from \cite[Proposition 4.46, Theorem 4.53]{BGV25} and definitions.
\end{rmk}

\begin{pro}
  \label{pro:HK-comp-rig}
  We have a monoidal equivalence $\pi \circ \rgama{\mathrm{HK},C} \circ \xi \simeq \rgama{\mathrm{rig}}$, where $\pi \colon \dcat_{(\varphi,N)}(\breve{K}) \to \dcat_{\varphi}(\breve{K})$ is forgetting the monodromy.
\end{pro}
\begin{proof}
The construction of (\ref{eq:HK-coh}) shows that we have a monoidal equivalence $\pi \circ \rgama{\mathrm{HK},C} \simeq \rgama{\mathrm{rig}} \circ \Psi$, where $\Psi \colon \rigmot(C) \to \agmot(\bar{k})$ is the motivic nearby cycle functor defined in \cite[Definition 4.15]{BGV25}. Note that the Monsky-Washnitzer functor $\xi$ is a section of $\Psi$. Therefore, the monoidal equivalence is clear.
\end{proof}

\subsection{The De Rham Fargues-Fontaine Cohomology}
\label{subsec:FFcoh}
We are now in a position to recall the construction of the de Rham-Fargues-Fontaine realization in \cite{LBV23}. We will use its covariant form as well.

The construction of the spreading-out functor \eqref{eq:spreading-0} $\mathcal{D}_0 \colon \rigmot(C^{\flat}) \to \rigmot(\mathbf{FF})$ shows that it sends compact objects to dualizable objects. On the other hand, as shown in \cite[Theorem 4.46]{LBV23}, the relative overconvergent de Rham realization, defined in \cite[\S 4]{LBV23},
\[
\rgama{\mathrm{dR}, \mathbf{FF}}^{\dagger} \colon \rigmot(\mathbf{FF})^{\opp} \to \qcoh(\mathbf{FF})
\]
sends dualizable objects to perfect complexes that lie in $\qcoh(\mathbf{FF})_{\omega}$. Therefore, using the dual trick, we get a covariant functor
\[
\rigmot(C)_{\omega} \simeq \rigmot(C^{\flat})_{\omega} \xrightarrow{\mathcal{D}_0\circ (-)^{\vee}} \rigmot(\mathbf{FF})_{\mathrm{dual}}^{\opp} \xrightarrow{\rgama{\mathrm{dR}, \mathbf{FF}}^{\dagger}} \qcoh(\mathbf{FF})_{\omega},
\]
where the first equivalence is the motivic tilting equivalence given in \cite{Vez19}. Its Ind-completion will be denoted by
\begin{equation}
  \label{eq:drFF}
\rgama{\mathrm{FF}} \colon \rigmot(C) \to \qcoh(\mathbf{FF}),  
\end{equation}
called the \myemph{de Rham-Fargues-Fontaine realization}.

In other words, the (dual of) de Rham-Fargues-Fontaine cohomology is the overconvergent de Rham cohomology of the spreading-out motive from $0$. We next show that it computes the rigid cohomology.

\begin{lem}
  \label{lem:inf-pullback-to-Y}
  Let $e \colon \mathcal{Y}_{(0,\infty)} \to \spa(\breve{K}, \mathcal{O}_{\breve{K}})$ defined in \eqref{eq:map-Y-breveK}. Then we have a commutative diagram
  \[
    %
\begin{tikzcd}
\rigmot(\breve{K})^{\varphi_{\omega}} \arrow[r, "e^*"]                                        & {\rigmot(\mathcal{Y}_{(0,\infty)})^{\varphi_{\omega}}} \\
{\rigmot(\mathcal{Y}_{(0,\infty]})^{\varphi_{\omega}}} \arrow[u, "\simeq"] \arrow[ru, "j^*"'] &                                                       
\end{tikzcd}
  \]
  in $\calg(\prlc)$, where vertical equivalence is Proposition \ref{pro:sp-out-curve} $(1)$ and the functor $j^{*}$ is the pullback along the inclusion.
\end{lem}
\begin{proof}
Recall from Remark \ref{rmk:factorization-e} the morphism $e$ has a factorization
  \[
e \colon \mathcal{Y}_{(0,\infty)} \xrightarrow{\iota} \mathcal{Y}_{(0,\infty]} \xrightarrow{\bar{e}} \spa(\breve{K}, \mathcal{O}_{\breve{K}}),
\]
and the closed immersion $c \colon \spa(\breve{K}, \mathcal{O}_{\breve{K}}) \to \mathcal{Y}_{(0,\infty]}$ is a section of $\bar{e}$. Thus, the commutativity follows from
\[
e^{*} \circ c^{*} \simeq \iota^{*} \circ \bar{e}^{*} \circ c^{*} \simeq \iota^{*} \circ (\bar{e} \circ c)^{*} \simeq \iota^{*}.
\]
\end{proof}

\begin{pro}
  \label{pro:drFF-rig}
  We have an equivalence $ \rgama{\mathrm{FF}} \circ \xi_C \simeq \mathcal{E} \circ \rgama{\mathrm{rig}}$ in $\calg(\prlc)$, where
  \[
\mathcal{E} \colon \dcat_{\varphi}(\breve{K}) \to \qcoh(\mathbf{FF})
\]
is defined in (\ref{eq:E-functor}).
\end{pro}
\begin{proof}
  As tilting equivalence of rigid analytic motives is compatible with the Monsky-Washnitzer functor (\cite[Corollary 3.9]{Vez19}), we can prove for replacing $C$ by $C^{\flat}$. Using the comparison of spreading out from $0$ and $\infty$, see Proposition \ref{pro:comp-spread-0-inf}, it suffices to show there is an equivalence
  \[
\rgama{\mathrm{dR}, \mathbf{FF}}^{\dagger} \circ \tilde{\mathcal{D}}_{\infty} \simeq \mathcal{E} \circ \rgama{\mathrm{rig}}.
\]
in $\calg(\prlc)$. To this end, we use the identifications $\qcoh(\mathbf{FF}) \simeq \qcoh(\mathcal{Y}_{(0,\infty)})^{\varphi_{\omega}}$ and $\rigmot(\mathbf{FF}) \simeq \rigmot(\mathcal{Y}_{(0,\infty)})^{\varphi_{\omega}}$, and prove the equivalence above by restricting to the compact part; namely, it suffices to prove the equivalence $\mathcal{E} \circ \rgama{\mathrm{dR}, \breve{K}}^{\dagger,\varphi_{\omega}} \simeq \rgama{\mathrm{dR}, \mathcal{Y}_{(0,\infty)}}^{\dagger,\varphi_{\omega}} \circ e^{*}$ on $\rigmot(\breve{K})_{\omega}^{\varphi_{\omega}}$, where $e^{*}$ is the pullback of motives along the morphism defined in (\ref{eq:map-Y-breveK}). This can be deduced from the commutativity of the following diagram:
\[
\begin{tikzcd}
\rigmot(\breve{K})^{\varphi_{\omega}}_{\omega} \arrow[r] 
  \arrow[dd, "e^*"', out=180, in=180]
  & \bdcat_{\varphi}(\breve{K}) \arrow[dd, "\mathcal{E}", out=0, in=0] 
\\
{\rigmot(\mathcal{Y}_{(0,\infty]})^{\varphi_{\omega}}_{\omega}} 
  \arrow[u, "\simeq"'] \arrow[d] \arrow[r] 
  & {\qcoh(\mathcal{Y}_{(0,\infty]})^{\varphi_{\omega}}_{\omega}} \arrow[d] \arrow[u]
\\
{\rigmot(\mathcal{Y}_{(0,\infty)})^{\varphi_{\omega}}_{\omega}} \arrow[r]
  & {\qcoh(\mathcal{Y}_{(0,\infty)})_{\omega}^{\varphi_{\omega}}}
\end{tikzcd}
\]
The middle squares are commutative due to \cite[Proposition 5.3]{LBV23}, and the leftmost part is commutative due to Lemma \ref{lem:inf-pullback-to-Y}, whose proof also works to show the commutativity of the rightmost part; see Remark \ref{rmk:factorization-e}.
\end{proof}


\subsection{Functors From Module Categories}

Before establishing the comparison between the Hyodo-Kato and the de Rham-Fargues-Fontaine cohomology theories, we review some abstract formalism from \cite{HA}.

\subsubsection{Free Commutative Algebras}
Let $\mathcal{C}$ be a symmetric monoidal $\infty$-category. The forgetful
functor $\theta\colon \calg(\mathcal{C}) \to \mathcal{C}$ admits a
left adjoint (e.g., see \cite[Proposition 3.1.3.13]{HA})
\[
S\colon \mathcal{C} \to \calg(\mathcal{C})
\]
called \myemph{free commutative algebra functor}. Moreover, for each
$X \in \mathcal{C}$, we have
\[
\theta(S(X))=\coprod_{n \ge 0} \sym^n(X).
\]

\begin{eg}
  \label{eg:chi1-free}
  The commutative algebra $\chi {\mathbbm 1}$ in $\agmot(\bar{k})$ plays an important role in the constructions of realization functors. In fact, it is a free commutative algebra $\chi {\mathbbm 1} \simeq S({\mathbbm 1}(-1)[-1])$ by \cite[Theorem 3.13]{Ayo24} or \cite[Proposition 2.63]{BGV25}.
\end{eg}


\subsubsection{Base Change of Algebras}
Let $F \colon \opr{\mathcal{C}} \to \opr{\mathcal{D}}$ be a functor in $\calg(\prstc)$. We next use the identifications (from \cite[Proposition 3.4.1.3]{HA})
\[
\calg(\prstc)_{\opr{\mathcal{C}}/-} \simeq \calg(\module{\opr{\mathcal{C}}}(\prstc)), \quad  \calg(\prstc)_{\opr{\mathcal{D}}/-} \simeq \calg(\module{\opr{\mathcal{D}}}(\prstc)).
\]
There is a natural adjunction studied in \cite[\S 4.5.3]{HA}:
\[
F^{*} \colon \calg(\prstc)_{\opr{\mathcal{C}}/-} \rightleftarrows \calg(\prstc)_{\opr{\mathcal{D}}/-} \colon F_{*}.
\]
More precisely, $F_{*}$ is induced by composing with $F$ and $F^{*}$ sends $\opr{\mathcal{C}} \to\opr{\mathcal{E}}$ to $\opr{\mathcal{D}} \otimes_{\opr{\mathcal{C}}} \opr{\mathcal{E}}$.

Under our assumptions, we can identify $\module{A}(\mathcal{C})$ and $\mathcal{D}$ as objects in $\calg(\prstc)_{\mathcal{C}/-}$ via the free module functor (\cite[\S 4.2.4]{HA}) $\mathrm{Free} \colon \mathcal{C} \to \module{A}(\mathcal{C})$ and the monoidal functor $F \colon \mathcal{C} \to \mathcal{D}$, respectively.

\begin{pro}
  \label{pro:mapsp-functors}
  Let $F \colon \mathcal{C} \to \mathcal{D}$ be a functor in $\calg(\prstc)$ and $A \in \calg(\mathcal{C})$. Then there is a homotopy equivalence
  \[
\map_{\calg(\prstc)_{\mathcal{C}/-}} (\module{A}(\mathcal{C}), \mathcal{D}) \simeq \map_{\calg(\mathcal{D})}(FA, {\mathbbm 1}_{\mathcal{D}}).
\]
Furthermore, if $A=S(t)$ is a free algebra in $\mathcal{C}$, then we have a homotopy equivalence
\[
\map_{\calg(\prstc)_{\mathcal{C}/-}} (\module{A}(\mathcal{C}), \mathcal{D}) \simeq \map_{\mathcal{D}}(F(t), {\mathbbm 1}_{\mathcal{D}}).
\]
\end{pro}
\begin{proof}
  In fact, here $\mathcal{D}= F_{*} \mathcal{D}$ in $\calg(\prstc)_{\mathcal{C}/-}$; thus, we have
\begin{align*}
  \map_{\calg(\prstc)_{\mathcal{C}/-}}(\module{A}(\mathcal{C}), \mathcal{D}) &\simeq \map_{\calg(\prstc)_{\mathcal{D}/-}}(F^{*} \module{A}(\mathcal{C}), \mathcal{D})\\
  &\simeq \map_{\calg(\prstc)_{\mathcal{D}/-}} (\module{FA}(\mathcal{D}), \mathcal{D})
\end{align*}
here the last homotopy equivalence follows from
\[
F^{*} \module{A}(\mathcal{C}) \simeq \mathcal{D}\otimes_{\mathcal{C}} \module{A}(\mathcal{C}) \simeq \module{FA}(\mathcal{D})
\]
by \cite[Theorem 4.8.4.6]{HA}. Therefore, we conclude the first assertion from the full faithfulness of the functor
\begin{align*}
  \calg(\mathcal{D}) &\hookrightarrow \calg(\prstc)_{\mathcal{D}/-}\\
  A&\mapsto \module{A}(\mathcal{D});
\end{align*}
see \cite[Corollary 4.8.5.21]{HA}. For the last assertion, we use the universal property of free algebras:
\begin{align*}
  \map_{\calg(\mathcal{D})}(FA, {\mathbbm 1}_{\mathcal{D}})&\simeq \map_{\calg(\mathcal{C})}(A, G({\mathbbm 1}_{\mathcal{D}}))\\
                                                           &\simeq \map_{\mathcal{C}}(t, G({\mathbbm 1}_{\mathcal{D}}))\\
  &\simeq \map_{\mathcal{D}}(F(t), {\mathbbm 1}_{\mathcal{D}})
\end{align*}
where $G$ is the right adjoint of $F$.
\end{proof}

\subsection{Monoidal Comparisons of Realization Functors}
\label{subsec:monoidal-comparison}
We have defined two kinds of realization functors: the de Rham-Fargues-Fontaine realization \eqref{eq:drFF} and the Hyodo-Kato realization \eqref{eq:HK-coh}. To compare these two realization functors, we can use the functor $\mathcal{E} \colon \dcat_{\varphi}(\breve{K}) \to \qcoh(\mathbf{FF})$ (see \eqref{eq:E-functor}) to refine the Hyodo-Kato realization so that it takes values in the same target category. More precisely, we can define a functor in $\calg(\prstc)$
\[
\mathcal{E}_N \colon \dcat_{(\varphi,N)}(\breve{K}) \to \dcat_{\varphi}(\breve{K}) \xrightarrow{\mathcal{E}} \qcoh(\mathbf{FF})
\]
where the first functor is forgetting\footnote{In fact, the vector bundles associated to $(\varphi,N)$-modules only depend on the underlying $\varphi$-modules; see \cite[Proposition 10.3.3]{thecurve}.} the monodromy operator.

\begin{thm}
  \label{thm:FFHK-comp}
 There is a unique (up to homotopy) functor $F \colon \rigmot(C) \to \qcoh(\mathbf{FF})$ in $\calg(\prstc)$ satisfying $F \circ \xi_{C} \simeq \mathcal{E} \circ \rgama{\mathrm{rig}}$. In particular, we have an equivalence $\rgama{\mathrm{FF}} \simeq \mathcal{E}_N \circ \rgama{\mathrm{HK}}$ in $\calg(\prstc)$.
\end{thm}
\begin{proof}
  The Monsky-Washnitzer functor $\xi_C \colon \agmot(\bar{k}) \to \rigmot(C)$ and the functor $\mathcal{E} \circ \rgama{\mathrm{rig}}$ define two objects $\rigmot(C)$ and $\qcoh(\mathbf{FF})$ in $\calg(\prstc)_{\agmot(\bar{k})/-}$. To show the uniqueness of $F$, it suffices to compute
  \begin{equation}
    \label{eq:pi0-functors}
    \pi_0 \map_{\calg(\prstc)_{\agmot(\bar{k})/-}} \left( \rigmot(C), \qcoh(\mathbf{FF}) \right)
  \end{equation}
  We now identify $\rigmot(C)$ with $\module{\chi {\mathbbm 1}}(\agmot(\bar{k}))$ and use the identification $\chi {\mathbbm 1} \simeq {\mathbbm 1} \oplus {\mathbbm 1}(-1)[-1]$ in Proposition \ref{pro:chi-unit-formula}. Using Proposition~\ref{pro:mapsp-functors}, we know that \eqref{eq:pi0-functors} is computed by
\begin{align*}
  \pi_0 \map_{\qcoh(\mathbf{FF})}(F({\mathbbm 1}(-1))[-1], \mathcal{O}_{\mathbf{FF}})&\simeq \pi_0 \map_{\qcoh(\mathbf{FF})} (\mathcal{E} \circ \rgama{\mathrm{rig}}({\mathbbm 1}(-1))[-1], \mathcal{O}_{\mathbf{FF}})\\
                                                                                     &\simeq \pi_0 \map_{\qcoh(\mathbf{FF})} (\mathcal{O}_{\mathbf{FF}}(-1)[-1], \mathcal{O}_{\mathbf{FF}})\\
  &\simeq \ext^1(\mathcal{O}_{\mathbf{FF}}(-1), \mathcal{O}_{\mathbf{FF}}) \simeq 0,
\end{align*}
where the last equivalence is Proposition \ref{pro:curve-twist}.

For the last assertion, it follows from the compatibilities of these two functors with the rigid realization; see Proposition \ref{pro:drFF-rig} and Proposition \ref{pro:HK-comp-rig}.
\end{proof}

\begin{rmk}
  \label{rmk:indpend-psu}
  In defining the Hyodo-Kato realization $\rgama{\mathrm{HK},C}$, we choose the canonical pseudo-uniformizer $p$ of $\mathcal{O}_{C}$. In fact, different choices of pseudo-uniformizer only change the monodromy operators. However, after applying the $\mathcal{E}$ functor, they yield the same solid quasi-coherent sheaves on $\mathbf{FF}$. Therefore, the comparison result in Theorem \ref{thm:FFHK-comp} is independent of the choice of pseudo-uniformizer.
\end{rmk}

\begin{rmk}
  \label{rmk:monoidal-eq-noncan}
  The previous theorem shows that we can find a monoidal natural
  equivalence between the de Rham Fargues-Fontaine realization and the
  Hyodo-Kato realization. But this natural equivalence is not
  unique: there is a big space
  \[
\pi_1 \map_{\calg(\prstc)_{\agmot(\bar{k})/-}}(\rigmot(C) , \qcoh(\mathbf{FF})) \simeq \pi_0 \map(\mathcal{O}_{\mathbf{FF}}(-1), \mathcal{O}_{\mathbf{FF}}) \simeq B^{\varphi=p}.
\]
of choices for the monoidal natural transformations.
\end{rmk}

\subsubsection{A New Filtration for Vector Bundles on $\mathbf{FF}$}
As an application of this comparison result, we can construct a new filtration on the vector bundles arose as the de Rham-Fargues-Fontaine cohomology.

Recall from \cite[\S 3.3]{BGV25} there is a bounded weight structure (which is a dual notion of $t$-structure defined in \cite{Bon10a}) on $\rigmot(C)_{\omega} \simeq \rigmot_{\mathrm{gr}}(C)_{\omega}$. Let $\mathcal{H}$ denote the heart of this weight structure. Then we have a functor in $\calg(\catinf^{\mathrm{ex}})$ (see \cite{Aokwt,BGV25})
\[
W \colon \rigmot(C)_{\omega} \to \bkch(\htcat \mathcal{H})
\]
where $\bkch(\htcat \mathcal{H})$ is the stable $\infty$-category of (homological) bounded chain complexes in the additive category $\htcat \mathcal{H}$. This functor is called the \myemph{weight complex functor}.

As a consequence of Theorem \ref{thm:FFHK-comp}, we know that the de Rham-Fargues-Fontaine realization factors through the weight complex functor on $\rigmot(C)$:

\begin{cor}
  \label{cor:wt-fil-ff}
The de Rham-Fargues-Fontaine realization on the compact motives over $C$ factors through the weight complex functor. In other words, we have a commutative diagram
  \[
\begin{tikzcd}
\rigmot(C)_{\omega} \arrow[r, "\rgama{\mathrm{FF}}"] \arrow[d, "W"']            & \mathbf{Perf}(\mathbf{FF}) \\
\bkch(\htcat \mathcal{H}) \arrow[ru, "\widetilde{\rgama{}}_{\mathrm{FF}}"', dashed] &                           
\end{tikzcd}
\]
in $\calg(\catinf^{\mathrm{ex}})$.
\end{cor}
\begin{proof}
The Hyodo-Kato realization functor has already factored through the weight complex functor by \cite[Theorem 4.53, Corollary 4.54]{BGV25}:
  \[
\rigmot(C)_{\omega} \xrightarrow{\rgama{\mathrm{HK}}} \bdcat_{(\varphi,N)}(\breve{K}) \to \bdcat_{\varphi}(\breve{K}).
\]
Therefore, we conclude from the comparison result (Theorem \ref{thm:FFHK-comp}).
\end{proof}

\begin{cor}
  \label{cor:fil-vbFF}
For every compact adic \'{e}tale motive $M$ over $C$, there is a convergent spectral sequence starting from the first page and degenerating at the second page:
  \[
E_{pq}^1=H_{q} \rgama{\mathrm{FF}}(W_pM) \Rightarrow H_{p+q} \rgama{\mathrm{FF}}(M)
\]
where $W_{\bullet}M$ is the weight complex of $M$. In particular, $H_{n} \rgama{\mathrm{FF}}(M)$ has a finite increasing filtration $\mathrm{Fil}_{\bullet}$ whose $i$-th graded piece is of slope $(i-n)/2$.
\end{cor}
\begin{proof}
  This spectral sequence is a special case of \cite[Proposition 1.2.2.14]{HA} applied to the naive filtration of the weight complex $W_{\bullet}M$ by truncations. We are left to show it degenerates at the second page and the slopes of $E_{pq}^{2}$.

  From \cite[Proposition 4.25]{BGV25}, each $W_pM$ has a form of $\xi_C Y_p$ where $Y_p$ is a Chow motive over $\bar{k}$. Using Proposition \ref{pro:drFF-rig}, we know that
  \[
H_q \rgama{\mathrm{FF}}(W_p M) \simeq \mathcal{E}(H_q \rgama{\mathrm{rig}}(Y_{p}))
\]
where $H_q \rgama{\mathrm{rig}}(Y_{p})$ is of weight $-q$ by \cite[Proposition 4.44]{BGV25}. Thus, $H_q \rgama{\mathrm{FF}}(W_pM)$ has the slope of $-q/2$ by Proposition \ref{pro:pure-isoc-half-slope} since $\mathcal{E}$ preserves the slopes by our definition of Tate twists of $\varphi$-modules. Therefore, $E_{pq}^{2}$ is also of slope $-q/2$; hence the differential map $E_{pq}^2 \to E_{p-2,q+1}^{2}$ vanishes by Proposition~\ref{pro:curve-twist}. This proves that the spectral sequence degenerates at the second page. In particular, the $i$-th graded piece of the induced filtration on $H_{n} \rgama{\mathrm{FF}}(M)$ is $E_{i,n-i}^{\infty}\simeq E_{i,n-i}^{2}$, whose slope is $(i-n)/2$.
\end{proof}

Given a compact motive $M$ over $C$, define $\mathcal{H}_{\mathrm{FF}}^i(M):= H_{-i} \rgama{\mathrm{FF}}(M^{\vee})$ because we are using the covariant cohomological realization functor. In the special case where $M$ is the associated motive of a smooth quasi-compact rigid analytic space over $C$, we write it simply by $\mathcal{H}_{\mathrm{FF}}^i(X)$. This is a vector bundle on the Fargues-Fontaine curve.

\begin{cor}
  \label{cor:fil-FFcoh-space}
  Let $X$ be a smooth quasi-compact rigid analytic variety over $C$. Then, for each $n \ge 0$, the vector bundle $\mathcal{H}_{\mathrm{FF}}^n(X)$ has a finite increasing filtration $\mathrm{Fil}_{\bullet}$ whose $i$-th graded piece is of slope $(i+n)/2$.
\end{cor}

\begin{rmk}
  \label{rmk:not-HN-fil}
  The filtration in Corollary \ref{cor:fil-FFcoh-space} is a new filtration on the vector bundles on the Fargues-Fontaine curve, different from the Harder–Narasimhan filtration as graded slopes are increasing. Even more, the set of slopes with respect to this new filtration does not agree with the set of Harder-Narasimhan slopes.
\end{rmk}

\subsection{The Uniqueness of Comparison Equivalences}
\label{sec:uniq-comp-isom}
As noted in Remark \ref{rmk:monoidal-eq-noncan}, the comparison natural isomorphism
\[
\rgama{\mathrm{FF}} \simeq \mathcal{E}_N \circ \rgama{\mathrm{HK}}
\]
is not unique; in fact, the space of such natural isomorphisms is large. To single out a unique good comparison isomorphism, we keep track of the Galois actions. More precisely, we work with a refined coefficient category: the $\infty$-category of $G_{\breve{K}}$-equivariant solid quasi-coherent sheaves on the Fargues-Fontaine curve; here $G_{\breve{K}}=\gal (C/\breve{K})$ is the absolute Galois group of $\breve{K}$.

We define this $\infty$-category as follows: the action of $G_{\breve{K}}$ on $A_{\mathrm{inf}}=W(\mathcal{O}_{C^{\flat}}) $ induces an action on the Fargues-Fontaine $\mathbf{FF}$ by its construction. Thus, pullbacks along these Galois-automorphisms give a $G_{\breve{K}}$-action on $\mathbf{FF}$, which induces a $G_{\breve{K}}$-action on the $\infty$-category $\qcoh(\mathbf{FF})$. In other words, we have a functor
\begin{align}
  \label{eq:G-on-qcFF}
  \begin{split}
      \subss{B}G_{\breve{K}} &\to \calg(\prstc)\\
  *&\mapsto \qcoh(\mathbf{FF})\\
  (g \in G_{\breve{K}})&\mapsto (\sigma_g^{*} \colon \qcoh(\mathbf{FF}) \to \qcoh(\mathbf{FF})).
  \end{split}
\end{align}
We then define the $\infty$-category of \myemph{$\boldsymbol{G_{\breve{K}}}$-equivariant solid quasi-coherent sheaves} on $\mathbf{FF}$ as the limit of the diagram above.

Next, we explain how to enhance the target of $\rgama{\mathrm{FF}}$ and $\mathcal{E}_n \circ \rgama{\mathrm{HK}}$.

\subsubsection{From $\boldsymbol{\varphi}$-modules to $\qcoh(\mathbf{FF})^{\htcat G_{\breve{K}}}$}

We can enrich the target of $\mathcal{E} $ in \eqref{eq:E-functor}: the definition of the morphism $e$ \eqref{eq:map-Y-breveK} shows $e$ is stable under $G_{\breve{K}}$, and the Frobenius map is compatible with the $G_{\breve{K}}$-action; that is, for every $g \in G_{\breve{K}} $, we have $e \circ \sigma_g = e$ and $\varphi \circ \sigma_g= \sigma_g \circ \varphi$. This implies that, for each $g \in G_{\breve{K}}$, there is a functor $\sigma^{*}_g \colon \mathbf{Perf}(\mathcal{Y}_{(0,\infty)})^{\varphi} \to \mathbf{Perf}(\mathcal{Y}_{(0,\infty)})^{\varphi}$ and $\mathcal{E}$ is stable under these functors, i.e., $\sigma^{*}_g \circ \mathcal{E} \simeq \sigma_g^{*}$. Thus, we get a refinement of $\mathcal{E}$ in $\calg(\prstc)$
\begin{equation}
  \label{eq:gal-E-functor}
  \mathcal{E}^{\mathrm{ari}} \colon \dcat_{\varphi}(\breve{K}) \to \qcoh(\mathbf{FF})^{\htcat G_{\breve{K}}}
\end{equation}
Clearly, the underlying solid quasi-coherent sheaves of $\mathcal{E}^{\mathrm{ari}}$ are given by the functor $\mathcal{E}$. As before, we define $\mathcal{E}_N^{\mathrm{ari}}:= \mathcal{E}^{\mathrm{ari}} \circ \pi$.

\subsubsection{Refinement of the de Rham Fargues-Fontaine Realization}

To get a Galois refinement of the de Rham-Fargues-Fontaine realization, we have to enhance each rigid analytic motive over $C$ into a $G_{\breve{K}}$-equivariant rigid analytic motive over $C$. To this end, we start from the natural functor
\begin{equation}
  \label{eq:K-to-galC}
  \rigmot_{\mathrm{gr}}(\breve{K}) \to \rigmot(C)^{\htcat G_{\breve{K}}}
\end{equation}
induced by the base change functor $\rigmot_{\mathrm{gr}}(K) \to \rigmot(C)$, here $\rigmot(C)^{\htcat G_{\breve{K}}}$ is the limit of the $G_{\breve{K}}$-action on $\rigmot(C)$. More precisely, the action of $G_{\breve{K}}$ on $C$ yields a $G_{\breve{K}}$-action on $\rigmot(C)$ that is a functor
\[
\subss{B} G_{\breve{K}} \to \calg(\prstc)
\]
sending the point to $\rigmot(C)$, similarly to \eqref{eq:G-on-qcFF}, and $\rigmot(C)^{\htcat G_{\breve{K}}}$ is the limit of this diagram. Since the base change $\rigmot_{\mathrm{gr}}(\breve{K}) \to \rigmot(C)$ is stable under this $G_{\breve{K}}$-action, we obtain (\ref{eq:K-to-galC}).

On the other hand, since $C/\breve{K}$ is totally ramified, the base change $\rigmot_{\mathrm{gr}}(\breve{K}) \xrightarrow{\simeq} \rigmot_{\mathrm{gr}}(C)$ is an equivalence by \cite[Proposition 3.23]{BKV25} and the latter category is equivalent to $\rigmot(C)$ via the natural embedding, as shown in \cite[Theorem 3.7.21]{AGV22}. Thus, the base change $\rigmot_{\mathrm{gr}}(K) \xrightarrow{\simeq} \rigmot(C)$ is indeed an equivalence of $\infty$-categories. Replacing $\rigmot_{\mathrm{gr}}(\breve{K})$ by $\rigmot(C)$ in \eqref{eq:K-to-galC} using this equivalence, we get a Galois-enhanced functor
\begin{equation}
  \label{eq:gal-enrich}
\alpha \colon  \rigmot(C) \to \rigmot(C)^{\htcat G_{\breve{K}}}.
\end{equation}

\begin{rmk}
  \label{rmk:galenhance}
  The intuition of (\ref{eq:gal-enrich}) is that, for every rigid analytic motive $M$ over $C$, we can find a unique (up to homotopy) model $\tilde{M}$ in $\rigmot_{\mathrm{gr}}(K)$; that is, we have an isomorphism $M \simeq \tilde{M}_C$ which gives the natural $G_{\breve{K}}$-action on $M$.

  More generally, let $F/L$ be a Galois extension of complete non-archimedean fields with residue fields $k_F$ and $k_L$ respectively. Let $I_{F/L}$ be the inertial group of this extension. Using the same argument, we can obtain the inertial enrichment:
  \[
\rigmot_{\mathrm{gr}}(F) \to \rigmot_{\mathrm{gr}}(F)^{\htcat I_{F/L}}.
\]
Here the $G_{F/L}$-action (hence $I_{F/L}$-action) on $\rigmot(F)$ can be restricted to $\rigmot_{\mathrm{gr}}(F)$ due to \cite[Proposition 3.1.13]{AGV22}.
\end{rmk}

\begin{lem}
  \label{lem:gal-from-model}
The composite functor $\rigmot(C) \xrightarrow{\alpha} \rigmot(C)^{\htcat G_{\breve{K}}} \to \rigmot(C)$, where $\alpha$ is defined in (\ref{eq:gal-enrich}) and the second map is the canonical projection, is the identity functor (up to homotopy).
\end{lem}
\begin{proof}
  This follows from the commutativity of the following diagram
  \[
\begin{tikzcd}
\rigmot_{\mathrm{gr}}(\breve{K}) \arrow[r] \arrow[d, "\simeq"'] \arrow[rd] & \rigmot(C)^{\htcat G_{\breve{K}}} \arrow[d] \\
\rigmot(C) \arrow[r,equal]                                           & \rigmot(C)                                 
\end{tikzcd}.
\]
\end{proof}

Note that the de Rham-Fargues-Fontaine realization
\[
\rgama{\mathrm{FF}} \colon \rigmot(C) \xrightarrow{\mathcal{D}_0} \rigmot(\mathcal{Y}_{(0,\infty)})^{\varphi_{\omega}} \to \qcoh(\mathbf{FF})
\]
is compatible with $G_{\breve{K}}$-actions. Then taking $G_{\breve{K}}$-homotopy fixed points, we get an enhanced realization functor:
\begin{equation}
  \label{eq:dRFF-gal}
  \rgama{\mathrm{FF}}^{\mathrm{ari}} \colon \rigmot(C) \xrightarrow{\alpha} \rigmot(C)^{\htcat G_{\breve{K}}} \xrightarrow{\rgama{\mathrm{FF}}^{\htcat G_{\breve{K}}}} \qcoh(\mathbf{FF})^{\htcat G_{\breve{K}}}
\end{equation}
In light of Lemma \ref{lem:gal-from-model}, the new realization functor $\rgama{\mathrm{FF}}^{\mathrm{ari}}$ computes the de Rham Fargues-Fontaine realization of motives over $C$, and it also can be compared with the Galois-enriched rigid realization:

\begin{thm}
  \label{thm:galFF-rig}
  \begin{enumerate}
  \item There is a unique (up to a unique natural isomorphism) functor
    \[
F \colon \rigmot(C) \to \qcoh(\mathbf{FF})
\]
in $\calg(\prstc)$ such that $F \circ \xi_{C} \simeq \mathcal{E}^{\mathrm{ari}} \circ \rgama{\mathrm{rig}}$, where $\mathcal{E}^{\mathrm{ari}}$ is defined in (\ref{eq:gal-E-functor}).

\item There are canonical monoidal equivalences
\begin{align*}
  \mathcal{E}_N^{\mathrm{ari}} \circ \rgama{\mathrm{HK}} \circ \xi_{C} &\simeq \mathcal{E}^{\mathrm{ari}} \circ \rgama{\mathrm{rig}}\\
  \rgama{\mathrm{FF}}^{\mathrm{ari}}&\simeq \mathcal{E}^{\mathrm{ari}}\circ \rgama{\mathrm{rig}}.
\end{align*}
In particular, $\rgama{\mathrm{FF}}^{\mathrm{ari}}$ and $\mathcal{E}_N^{\mathrm{ari}} \circ \rgama{\mathrm{HK}}$ are two objects in $\calg(\prstc)_{\agmot(\bar{k})/-}$ via these monoidal equivalences.
    
\item There is a unique monoidal natural isomorphism $\rgama{\mathrm{FF}}^{\mathrm{ari}} \simeq \mathcal{E}^{\mathrm{ari}}_N \circ\rgama{\mathrm{HK}}$ in $\calg(\prstc)_{\agmot(\bar{k})/-}$.
\end{enumerate}
\end{thm}
\begin{proof}
\begin{enumerate}
\item The proof is similar to Theorem \ref{thm:FFHK-comp}: we have a homotopy equivalence
  \[
\map_{\calg(\prstc)_{\agmot(\bar{k})/-}}(\rigmot(C), \qcoh(\mathbf{FF})^{\htcat G_{\breve{K}}}) \simeq \map_{\qcoh(\mathbf{FF})^{\htcat G_{\breve{K}}}}( \mathcal{O}_{\mathbf{FF}}(-1)[-1], \mathcal{O}_{\mathbf{FF}}).
\]
Since $\qcoh(\mathbf{FF})$ is $\Q$-linear, the space is connected, and the fundamental group of this space is $(B^{G_{\breve{K}}})^{\varphi=p}$ (see also Remark \ref{rmk:monoidal-eq-noncan}), which is $0$ by \cite[Corollaire 10.2.8]{thecurve}.

\item After identifying $\rigmot(C)$ with $\rigmot_{\mathrm{gr}}(\breve{K})$, Theorem \ref{thm:FFHK-comp} shows we have a commutative diagram
  \[
\begin{tikzcd}
\rigmot_{\mathrm{gr}}(\breve{K}) \arrow[r, "\simeq"] \arrow[d, "{\pi \circ \rgama{\mathrm{HK}, \breve{K}}}"'] & \rigmot(C) \arrow[d, "\rgama{\mathrm{FF}}"] \\
\dcat_{\varphi}(\breve{K}) \arrow[r, "\mathcal{E}"']                                                          & \qcoh(\mathbf{FF})                         
\end{tikzcd}
\]
The right vertical functor is acted on by $G_{\breve{K}}$ via base change, and the action fixes the left vertical functor; thus, this gives an equivalence $\rgama{\mathrm{FF}}^{\mathrm{ari}} \simeq \mathcal{E}^{\mathrm{ari}} \circ \pi \circ \rgama{\mathrm{HK}}$ in $\calg(\prstc)$. By the definition, we clearly have
\[
\mathcal{E}^{\mathrm{ari}}_N \circ  \rgama{\mathrm{HK}} \circ \xi_{C} \simeq \mathcal{E}^{\mathrm{ari}} \circ \rgama{\mathrm{rig}}.
\]

\item This is a direct consequence of $(1)$ and $(2)$.
\end{enumerate}
\end{proof}

\begin{rmk}
  \label{rmk:monodromy}
  As we saw, the monodromy operator of Hyodo-Kato realization played no role throughout: we assign a solid quasi-coherent sheaf on the Fargues-Fontaine curve via the underlying $\varphi$-modules. In \cite[\S 10.3]{thecurve}, Fargues and Fontaine gave a direct construction
  \begin{equation}
    \label{eq:FF-mono}
\tilde{\mathcal{E}}_N \colon \dcat_{(\varphi,N)}(\breve{K}) \to \qcoh(\mathbf{FF})^{\htcat G_{\breve{K}}}    
  \end{equation}
with a natural isomorphism $\mathrm{pr} \circ \mathcal{E}^{\mathrm{ari}} \circ \pi \xrightarrow{\simeq} \mathrm{pr} \circ \mathcal{E}_N^{\mathrm{ari}}$, where $\mathrm{pr} \colon \qcoh(\mathbf{FF})^{\htcat G_{\breve{K}}} \to \qcoh(\mathbf{FF})$ is the canonical projection; see \cite[\S 10.3.2, (4)]{thecurve}. However, this natural isomorphism is not compatible with the base change of $\qcoh(\mathbf{FF})$ along the $G_{\breve{K}}$-action we defined here.
\end{rmk}

\begin{rmk}
  \label{rmk:galois-filtration}
Using the Galois-enriched comparison in Theorem \ref{thm:galFF-rig}, we can also know it factors through the weight complex functor factors. So we have a refinement of Corollary \ref{cor:fil-vbFF}; in other words, the finite filtration is also compatible with the $G_{\breve{K}}$-action.
\end{rmk}

\subsection{The Fargues-Fontaine Cohomology Via the D\'{e}calage Functor}
\label{subsec:FF-BMS}
There is an another Fargues-Fontaine cohomology, defined via the D\'{e}calage functor, was originally defined by Le Bras in \cite{LB18} and later developed by Bosco using condensed mathematics in \cite{Bos23}.

In the final part of this section, we show that this cohomology is motivic—that is, it extends to a functor on $\rigmot(C)$. Le Bras previously showed that it is defined on the $\infty$-category of effective motives $\rigmot^{\mathrm{eff}}(C)$ in \cite{LB18}.

For every compact sub-interval of $(0,\infty)$ with rational endpoints, and every smooth dagger variety $X$ over $C$, Bosco defined in \cite[Definition 2.38, Proposition 2.42]{Bos23} a solid $(B_I,\Z)$-module $\rgama{B_I}(X)$, in the sense of \cite[Definition 3.20 and Theorem 2.7]{Andreqcoh21}, using the d\'{e}calage functor; see also \cite[Definition 6.12, Remark 6.13]{Bos23}. This construction is referred to as the \myemph{$\boldsymbol{B_I}$-cohomology} of $X$. We prove this cohomology is already motivic:

\begin{pro}
  \label{pro:BI-motivic}
  For every compact sub-interval $I \subseteq (0,\infty)$, the $B_I$-cohomology of dagger varieties over $C$ extends to a functor
  \[
\rgama{B_I}(-) \colon \rigmot(C)\simeq \rigmot^{\dagger}(C) \to \dcat_{\solid}(B_I,\Z)^{\opp},
\]
where $\rigmot^{\dagger}(C)$ is the $\infty$-category of dagger motives over $C$, as defined in \cite{MWrig, LBV23}, and $\dcat_{\solid}(B_I,\Z)$ is the $\infty$-category of solid $(B_I,\Z)$-modules. Here the first equivalence is given by \cite[Theorem 4.23]{MWrig}.
\end{pro}
\begin{proof}
  Let $\mathbf{FSch}^{\dagger, \mathrm{ss}}_{\mathcal{O}_C}$ denote the category of semi-stable weak formal schemes over $\mathcal{O}_{C}$; see \cite[\S 2.3.1]{CN20}. Then for every $\mathfrak{X} \in \mathbf{FSch}^{\dagger, \mathrm{ss}}_{\mathcal{O}_{C}}$, by \cite[(4.50) together with Remark 6.13]{Bos23}, there is a natural isomorphism
  \begin{equation}
    \label{eq:BI-HK}
    \rgama{B_I}(\mathfrak{X}^{\mathrm{rig} \dagger}) \simeq \left( \rgama{\mathrm{HK}}(\mathfrak{X}^{\mathrm{rig} \dagger}) \otimes_{\breve{K}}^{\solid} B_{\mathrm{log},I}\right)^{N=0},
  \end{equation}
  where $B_{\mathrm{log},I}$ is the condensed period ring by taking $X= \spa(C)$ in \cite[Definition 2.27]{Bos23}, and $\mathfrak{X}^{\mathrm{rig} \dagger}$ is the dagger generic fiber of $\mathfrak{X}$; see \cite{LM13}. We next deduce from the comparison \eqref{eq:BI-HK} that the $B_I$-cohomology
  \[
\mathfrak{X} \mapsto \rgama{B_I}(\mathfrak{X}^{\mathrm{rig} \dagger})
\]
has the rig-\'{e}tale descent and $\mathbb{A}^1$-invariance. Indeed, as shown in \cite[Proposition 3.8]{BKV25}, the overconvergent Hyodo-Kato cohomology
  \[
\mathfrak{X} \mapsto \rgama{\mathrm{HK}}(\mathfrak{X}^{\mathrm{rig} \dagger})
\]
has the rig-\'{e}tale descent and $\mathbb{A}^1$-invariance. Thus, the $\mathbb{A}^1$-invariance is immediate. For the rig-\'{e}tale descent, we need to show $(-\otimes^{\solid}_{\breve{K}} B_{\mathrm{log},I})^{N=0}$ commutes with descent limits. This follows from \cite[Corollary A.67 (ii)]{Bos23a} together with \cite[(3.16)]{Bos23}, where the assumptions hold in our case due to the finiteness of overconvergent Hyodo-Kato cohomologies.

For the Tate-stability, it suffices to compute $H^1_{B_I}(\mathbb{G}_m^{\dagger})$. Using \eqref{eq:BI-HK} and \cite[Lemma 7.6]{Bos23}---where finiteness of overconvergent Hyodo-Kato cohomologies plays a role---we know that
\[
H^1_{B_I}(\mathbb{G}_m^{\dagger}) \simeq H^1_{\mathrm{HK}}(\mathbb{G}_{m, \bar{k}^0}) \otimes_{\breve{K}} B_I \simeq B_I,
\]
is invertible in $\dcat_{\solid}(B_I, \Z)$.

Finally, as in \cite[Proposition 3.8]{BKV25}, we can apply the dagger analogue\footnote{The proof of loc.cit. also holds thanks to \cite[Proposition 2.13]{CN20}.} of \cite[Proposition 2.32]{BKV25} to conclude the result.
\end{proof}

From the comparison \eqref{eq:BI-HK} and the boundedness of the Hyodo-Kato cohomology (see \cite[Theorem 3.14 (ii)]{Bos23}), we know that the motivic realization in Proposition \ref{pro:BI-motivic} restricts to a realization
\[
\rgama{B_I} \colon \rigmot(C)_{\omega} \to \bdcat_{\solid}(B_I,\Z)^{\opp}.
\]
Therefore, as before, we get the covariant version of this motivic realization
\[
\rgama{B_I} \colon \rigmot(C) \to \dcat_{\solid}(B_I,\Z)
\]
in $\calg(\prlc)$ by taking duals.

According to \cite[Lemma 6.14]{Bos23}, the family of motivic realization functors $\rgama{B_I}$ (restricted to compact parts) are compatible with inclusions and Frobenius maps between $B_I$'s. Applying a new $\mathcal{E}$-functor
\[
\mathcal{E}_B \colon \left( \lim_{I \subseteq (0,\infty)} \dcat_{\solid}(B_I, \Z) \right)^{\varphi_{\omega}} \to \qcoh(\mathbf{FF})
\]
as defined in \cite[\S 5]{LB18} and \cite[\S 6.2]{Bos23}, to the family of these realization functors $\rgama{B_I}$, we have a new \myemph{Fargues-Fontaine cohomology}:
\begin{equation}
  \label{eq:LBBFF}
\widetilde{\mathrm{R} \boldsymbol{\Gamma}}_{\mathrm{FF}} \colon \rigmot(C) \to \qcoh(\mathbf{FF})
\end{equation}
which is in $\calg(\prlc)$ since it sends compact motives into perfect complexes on the Fargues-Fontaine curve by \cite[Theorem 6.17 (i)]{Bos23}.

Let us explain the source of $\mathcal{E}_B$ above: we first take the limit along inclusions between $\dcat_{\solid}(B_I,\Z)$ in $\calg(\prstc)$; then there is an induced Frobenius endofunctor
\begin{equation}
  \label{eq:frob-CoadB}
  \varphi \colon \lim_{I \subseteq (0,\infty)} \dcat_{\solid}(B_I, \Z) \to \lim_{I \subseteq (0,\infty)} \dcat_{\solid}(B_I, \Z)
\end{equation}
induces by a Frobenius shift
\[
\lim_{I \subseteq (0,\infty)} \dcat_{\solid}(B_I, \Z) \to \dcat_{\solid}(B_{p \inv I}, \Z) \xrightarrow{\varphi} \dcat_{\solid}(B_I,\Z).
\]
Finally, we define the source of $\mathcal{E}_{B}$ as the equalizer of the identify functor and this Frobenius endofunctor \eqref{eq:frob-CoadB} in $\calg(\prstc)$.

This motivic Fargues-Fontaine cohomology \eqref{eq:LBBFF} agrees (non-canonically) with the de Rham-Fargues-Fontaine cohomology by their comparison with the Hyodo-Kato cohomology:

\begin{pro}[Comparison of Fargues-Fontaine Cohomologies]
  \label{pro:comp-FFcoh}
  There is a monoidal equivalence $\rgama{\mathrm{FF}} \simeq \widetilde{\mathrm{R} \boldsymbol{\Gamma}}_{\mathrm{FF}}$ in $\calg(\prstc)$.
\end{pro}
\begin{proof}
By \cite[Theorem 6.17]{Bos23}\footnote{The Hyodo-Kato cohomology in loc.cit. agrees with the motivic Hyodo-Kato realization by \cite[Proposition 3.30]{BKV25} and \cite[Corollary 4.58]{BGV25}.}, there is a natural isomorphism $\widetilde{\mathrm{R} \boldsymbol{\Gamma}}_{\mathrm{FF}} \simeq \mathcal{E}_N \circ \rgama{\mathrm{HK}}$. Therefore, we conclude from the comparison in Theorem \ref{thm:FFHK-comp}.
\end{proof}


\titleformat{\section}[display]
{\normalfont\Large\bfseries\filcenter}
{\thesection}%
{1pc}
{\titlerule[2pt]
  \vspace{1pc}%
  \huge}

\printbibliography[heading=bibintoc]
\end{document}


%% file: Notations-min.tex
\newcommand{\calg}{\mathbf{CAlg}} 
\newcommand{\opp}{\mathrm{op}} 
\newcommand{\shv}{\mathbf{Shv}} 
\newcommand{\catinf}{\mathbf{\boldsymbol{\mathcal{C}}at}_{\infty}} 
\newcommand{\pr}{\boldsymbol{\mathcal{P}}r} 
\newcommand{\prl}{\mathbf{\pr}^{\mathrm{L}}} 
\newcommand{\prlc}{\prl_{\omega}} 
\newcommand{\prst}{\mathbf{\pr}^{\mathrm{st}}} 
\newcommand{\prstc}{\prst_{\omega}} 
\newcommand{\opr}[1]{ #1^{\otimes}} 
\newcommand{\module}[1]{\mathbf{Mod}_{#1}} 
\DeclareMathOperator{\frc}{Frac} 
\DeclareMathOperator{\ext}{Ext} 
\DeclareMathOperator{\sym}{Sym} 
\newcommand{\inv}{^{-1}} 
\newcommand{\kch}{\boldsymbol{\mathcal{K}}} 
\newcommand{\bkch}{\kch^{\mathrm{b}}} 
\newcommand{\dcat}{\mathcal{D}}
\newcommand{\bdcat}{\dcat^{\mathrm{b}}}
\DeclareMathOperator{\gal}{Gal} 
\newcommand{\Z}{\mathbb{Z}} 
\newcommand{\Q}{\mathbb{Q}} 
\DeclareMathOperator{\spf}{Spf} 
\newcommand{\ets}{_{\text{\'{e}t}}} 
\newcommand{\qcoh}{\mathbf{QCoh}} 
\newcommand{\rgama}[1]{\mathrm{R}\boldsymbol{\Gamma}_{#1}} 
\DeclareMathOperator{\htcat}{h} 
\newcommand{\subss}[1]{ #1_{\bullet}} 
\newcommand{\map}{\mathrm{Map}}  
\DeclareMathOperator{\spa}{Spa} 
\newcommand{\solid}{{\scalebox{0.5}{$\square$}}} 
\newcommand{\agmot}{\mathbf{DA}}
\newcommand{\fagmot}{\mathbf{FDA}}
\newcommand{\rigmot}{\mathbf{RigDA}}